\documentclass[english,10pt]{amsart}

\usepackage[english]{babel}

\usepackage{amscd,amssymb,amsfonts,amsmath}
\usepackage{tikz-cd}

\usepackage{bm}
\usepackage{graphics}
\usepackage{epsfig}
\usepackage{mathrsfs}
\usepackage{xcolor}
\DeclareMathAlphabet{\mathscrbf}{OMS}{mdugm}{b}{n}

\definecolor{violet}{rgb}{0.0,0.2,0.7}
\definecolor{rouge2}{rgb}{0.8,0.0,0.2}
\usepackage{hyperref}
\hypersetup{
    unicode=false,          
    pdftoolbar=true,        
    pdfmenubar=true,        
    pdffitwindow=false,     
    pdfstartview={FitH},    
    pdftitle={},    
    pdfauthor={},     
    colorlinks=true,       
   linkcolor=violet,          
    citecolor=rouge2,        
    filecolor=black,      
    urlcolor=cyan}           

\usepackage[text={6.7in,9.2in},centering]{geometry}

\renewcommand{\phi}{\varphi}
\newcommand{\into}{\hookrightarrow}
\newcommand{\map}{\dashrightarrow}

\newcommand{\wt}{\widetilde}
\newcommand{\wh}{\widehat}
\newcommand{\wb}{\overline}

\renewcommand{\le}{\leqslant}
\renewcommand{\ge}{\geqslant}

\newcommand{\sE}{\mathscr{E}}

\newcommand{\sG}{\mathscr{G}}
\newcommand{\sH}{\mathscr{H}}
\newcommand{\sI}{\mathscr{I}}

\newcommand{\sL}{\mathscr{L}}
\newcommand{\sM}{\mathscr{M}}
\newcommand{\sN}{\mathscr{N}}
\newcommand{\sO}{\mathscr{O}}

\newtheorem{thm}{Theorem}[section]

\newtheorem{lemma}[thm]{Lemma}
\newtheorem{cor}[thm]{Corollary}

\newtheorem{prop}[thm]{Proposition}

\newtheorem*{thm*}{Theorem}
\theoremstyle{definition}
\newtheorem{defn}[thm]{Definition}

\newtheorem{notation}[thm]{Notation}

\newtheorem{defn-thm}[thm]{Definition-Theorem} 
\newtheorem{defn-lemma}[thm]{Definition-Lemma}

\theoremstyle{remark}
\newtheorem{claim}[thm]{Claim}
\newtheorem{fact}[thm]{Fact}

\newtheorem*{not-and-def}{Notation and definitions}
 
\newtheorem{rem}[thm]{Remark}

\numberwithin{equation}{section}

\def\factor#1.#2.{\left. \raise 2pt\hbox{$#1$} \right/\hskip -2pt\raise -2pt\hbox{$#2$}}

\begin{document} 

\title[Codimension one foliations with trivial canonical class on singular spaces II]{Codimension one foliations with numerically trivial canonical class on singular spaces II}

\author{St\'ephane \textsc{Druel}}

\address{St\'ephane Druel: Univ. Lyon, CNRS, Universit\'e Claude Bernard Lyon 1, UMR 5208, Institut Camille Jordan, F-69622 Villeurbanne, France.} 

\email{stephane.druel@math.cnrs.fr}

\author{Wenhao \textsc{Ou}}

\address{Wenhao Ou: Academy of Mathematics and Systems Science, Chinese Academy of Sciences, 55 ZhongGuanCun East Road, Beijing, 100190, China.} 

\email{wenhaoou@amss.ac.cn}



\begin{abstract}
In this article, we give the structure of codimension one foliations with canonical singularities and numerically trivial canonical class on varieties with klt singularities. Building on recent works of Spicer, Cascini - Spicer and Spicer - Svaldi, we then describe the birational geometry of rank two foliations with canonical singularities and canonical class of numerical dimension zero on complex projective threefolds. 
\end{abstract}

\maketitle
{\small\tableofcontents}

\section{Introduction}

In the last few decades, much progress has been made in the classification of complex projective (singular) varieties. The general viewpoint is that complex projective varieties should be classified according to the behavior of their canonical class. Similar ideas have been successfully applied to the study of global properties of holomorphic foliations. This led, for instance, to the birational classification of foliations by curves on surfaces (\cite{brunella}, \cite{mcquillan08}), generalizing most of the important results of the Enriques-Kodaira classification. 
However, it is well known that the abundance conjecture fails already in ambiant dimension two.
In very recent works, a foliated analogue of the minimal model program is established for rank two foliations on projective threefolds (\cite{spicer}, \cite{cascini_spicer} and \cite{spicer_svaldi}).

\medskip

The foliated analogue of the minimal model program aims in particular to reduce the birational study of mildly singular foliations with numerical dimension zero on complex projective manifolds to the study of associated minimal models, that is, mildly singular foliations with numerically trivial canonical class on klt spaces. In \cite{lpt}, motivated by these developments, the authors describe the structure of codimension one foliations with canonical singularities (we refer to Section 2 for this notion) and numerically trivial canonical class on complex projective manifolds. This result was extended by the first-named author to the setting of projective varieties with canonical singularities in \cite{cd1fzerocan}.
However, from the point of view of birational classification of foliations, this class of singularities is inadequate.
The main result of this paper settles this problem in full generality.

\begin{thm}\label{thm_intro:main}
Let $X$ be a normal complex projective variety with klt singularities, and let $\sG$ be a codimension one
foliation on $X$ with canonical singularities. Suppose furthermore that $K_\sG\equiv 0$.
Then one of the following holds.
\begin{enumerate}
\item There exist a smooth complete curve $C$, a complex projective variety $Y$ with canonical singularities and 
$K_Y \sim_\mathbb{Z}0$, as well as a quasi-\'etale cover $f \colon Y \times C \to X$ such that $f^{-1}\sG$ is induced by the projection $Y \times C \to C$.
\item There exist complex projective varieties $Y$ and $Z$ with canonical singularities, as well as a quasi-\'etale cover 
$f \colon Y \times Z \to X$ and a foliation $\sH\cong \sO_Z^{\, \dim Z -1}$ on $Z$ 
such that $f^{-1}\sG$ is the pull-back of $\sH$ via the projection $Y \times Z \to Z$. In addition, we have
$K_Y \sim_\mathbb{Z} 0$, $Z$ is an equivariant compactification of a commutative algebraic group of dimension at least $2$
and $\sH$ is induced by a codimension one Lie subgroup.
\end{enumerate}
\end{thm}

As an immediate consequence, we prove the abundance conjecture in this setting.

\begin{cor}
Let $X$ be a normal complex projective variety with klt singularities, and let $\sG$ be a codimension one
foliation on $X$ with canonical singularities. If $K_\sG\equiv 0$, then $K_\sG$ is torsion.
\end{cor}

Together with the foliated analogue of the minimal program for rank two foliations with F-dlt singularities on normal projective threefolds (see \cite{spicer}, \cite{cascini_spicer} and \cite{spicer_svaldi}), we obtain the following result.

\begin{cor}\label{cor_intro:main}
Let $X$ be a normal complex projective threefold, and let $\sG$ be a codimension one
foliation on $X$ with canonical singularities. Suppose furthermore that $\nu(K_\sG) = 0$.
Then one of the following holds.
\begin{enumerate}
\item There exist a smooth complete curve $C$, a complex projective variety $Y$ with canonical singularities and 
$K_Y \sim_\mathbb{Z}0$, as well as a generically finite rational map $f \colon Y \times C \map X$ such that $f^{-1}\sG$ is induced by the projection $Y \times C \to C$.
\item There exist a smooth complete curve $C$ of genus one, a complex projective surface $Z$ with canonical singularities, as well as a generically finite rational map $f \colon C \times Z \map X$ and a foliation $\sH\cong \sO_Z$ on $Z$ such that $f^{-1}\sG$ is the pull-back of $\sH$ via the projection $C \times Z \to Z$. In addition, $Z$ is an equivariant compactification of a commutative algebraic group and $\sH$ is induced by a one dimensional Lie subgroup.
\end{enumerate}
\end{cor}

This paper is a sequel to the article \cite{cd1fzerocan} by the first-named author and follows the same general strategy. The main new (crucial) ingredients are Propositions \ref{prop:transverse} and \ref{prop:general_type}. In order to prove these results, one needs to extend the Baum-Bott formula and the Camacho-Sad formula to surfaces with klt singularities. 

\subsection*{Structure of the paper} In section 2,  we recall the definitions and basic properties of foliations. We also establish some Bertini-type results. In section 3, we extend a number of earlier results to the context of quasi-projective varieties with quotient singularities. In particular, we extend the Baum-Bott formula as well as the Camacho-Sad formula to this context.
Section 4 prepare for the proof of our main result. We confirm the Ekedahl-Shepherd-Barron-Taylor conjecture for mildly singular codimension one foliations with trivial canonical class on projective varieties $X$ with $\nu(X)=-\infty$ in Section 5, and then on those with $\nu(X)=1$ in Section 6. In section 7, we address codimension one foliations with numerically  trivial  canonical  class  defined  by  closed  twisted rational $1$-forms. Section 8 is devoted to the proof of Theorem \ref{thm_intro:main} and Corollary \ref{cor_intro:main}.

\subsection*{Acknowledgements} We would like to thank Frank Loray, Jorge  V.  Pereira  and
Fr\'ed\'eric Touzet for answering our questions concerning their paper \cite{lpt}. We are very grateful to Jorge  V.  Pereira who explained to us the proof of Proposition \ref{prop:bertini}. 

The first-named author was partially supported by the ERC project ALKAGE (ERC grant Nr 670846), the CAPES-COFECUB project Ma932/19 and the ANR project Foliage (ANR grant Nr ANR-16-CE40-0008-01).

\subsection*{Global conventions} Throughout the paper we work over the complex number field.  

Given a variety $X$, we denote by $X_\textup{reg}$ its smooth locus.

We will use the notions of terminal, canonical, klt, and lc singularities for pairs
without further explanation or comment and simply refer to \cite[Section 2.3]{kollar_mori} for a discussion and for their precise definitions.

Given a normal variety $X$, $m\in \mathbb{N}_{>0}$, and coherent sheaves $\sE$ and $\sG$ on $X$, write
$\sE^{[m]}:=(\sE^{\otimes m})^{**}$, 
$\det\sE:=(\Lambda^{\textup{rank} \,\sE}\sE)^{**}$, and 
$\sE \boxtimes \sG := (\sE \otimes\sG)^{**}.$
Given any morphism $f \colon Y \to X$, write 
$f^{[*]}\sE:=(f^*\sE)^{**}.$

\section{Foliations}

In this section, we have gathered a number of results and facts concerning foliations which will later be used in the proofs.

\begin{defn}

A \emph{foliation} on  a normal variety $X$ is a coherent subsheaf $\sG\subseteq T_X$ such that
\begin{enumerate}
\item $\sG$ is closed under the Lie bracket, and
\item $\sG$ is saturated in $T_X$. In other words, the quotient $T_X/\sG$ is torsion-free.
\end{enumerate}

The \emph{rank} $r$ of $\sG$ is the generic rank of $\sG$.
The \emph{codimension} of $\sG$ is defined as $q:=\dim X-r$. 

\medskip

The \textit{canonical class} $K_{\sG}$ of $\sG$ is any Weil divisor on $X$ such that  $\sO_X(-K_{\sG})\cong \det\sG$. 

\medskip

Let $X^\circ \subseteq X_{\textup{reg}}$ be the open set where $\sG_{|X_{\textup{reg}}}$ is a subbundle of $T_{X_{\textup{reg}}}$. The \textit{singular locus} of $\sG$ is defined to be $X \setminus X^\circ$.
A \emph{leaf} of $\sG$ is a maximal connected and immersed holomorphic submanifold $L \subset X^\circ$ such that
$T_L=\sG_{|L}$. A leaf is called \emph{algebraic} if it is open in its Zariski closure.

The foliation $\sG$ is said to be \emph{algebraically integrable} if its leaves are algebraic.
\end{defn}

\subsection{Foliations defined by $q$-forms} \label{q-forms}
Let $\sG$ be a codimension $q$ foliation on an $n$-dimensional normal variety $X$.
The \emph{normal sheaf} of $\sG$ is $\sN:=(T_X/\sG)^{**}$.
The $q$-th wedge product of the inclusion
$\sN^*\into \Omega^{[1]}_X$ gives rise to a non-zero global section 
 $\omega\in H^0(X,\Omega^{q}_X\boxtimes \det\sN)$
 whose zero locus has codimension at least two in $X$. 
Moreover, $\omega$ is \emph{locally decomposable} and \emph{integrable}.
To say that $\omega$ is locally decomposable means that, 
in a neighborhood of a general point of $X$, $\omega$ decomposes as the wedge product of $q$ local $1$-forms 
$\omega=\omega_1\wedge\cdots\wedge\omega_q$.
To say that it is integrable means that for this local decomposition one has 
$d\omega_i\wedge \omega=0$ for every  $i\in\{1,\ldots,q\}$. 
The integrability condition for $\omega$ is equivalent to the condition that $\sG$ 
is closed under the Lie bracket.

Conversely, let $\sL$ be a reflexive sheaf of rank $1$ on $X$, and let
$\omega\in H^0(X,\Omega^{q}_X\boxtimes \sL)$ be a global section
whose zero locus has codimension at least two in $X$.
Suppose that $\omega$  is locally decomposable and integrable.
Then  the kernel
of the morphism $T_X \to \Omega^{q-1}_X\boxtimes \sL$ given by the contraction with $\omega$
defines 
a foliation of codimension $q$ on $X$. 
These constructions are inverse of each other. 

\subsection{Foliations described as pull-backs} \label{pullback_foliations}
Let $X$ and $Y$ be normal varieties, and let $\varphi\colon X\map Y$ be a dominant rational map that restricts to a morphism $\varphi^\circ\colon X^\circ\to Y^\circ$,
where $X^\circ\subseteq X$ and  $Y^\circ\subseteq Y$ are smooth open subsets.

Let $\sG$ be a codimension $q$ foliation on $Y$. Suppose that the restriction $\sG^\circ$ of $\sG$ to $Y^\circ$ is
defined by a twisted $q$-form
$\omega_{Y^\circ}\in H^0(Y^\circ,\Omega^{q}_{Y^\circ}\otimes \det\sN_{\sG^\circ}).$
Then $\omega_{Y^\circ}$ induces a non-zero twisted $q$-form 
$$\omega_{X^\circ}:= d\phi^\circ(\omega_{Y^\circ})\in 
H^0\big(X^\circ,\Omega^{q}_{X^\circ}\otimes (\varphi^\circ)^*({\det\sN_\sG}_{|Y^\circ})\big)$$ which defines a codimension $q$ foliation $\sE^\circ$ on $X^\circ$. 
\emph{The pull-back $\varphi^{-1}\sG$ of $\sG$ via $\varphi$} is the foliation on $X$ 
whose restriction to $X^\circ$ is $\sE^\circ$. We will also write $\sG_{|X}$ instead of $\varphi^{-1}\sG$.

\subsection{Singularities of foliations} Recently, notions of singularities coming from the minimal model program have shown to be very useful  when studying birational geometry of foliations. We refer the reader to \cite[Section I]{mcquillan08} for an in-depth discussion.
Here we only recall the notion of canonical foliation following McQuillan (see \cite[Definition I.1.2]{mcquillan08}).

\begin{defn}\label{definition:canonical_singularities}
Let $\sG$ be a foliation on a normal complex variety $X$. Suppose that $K_\sG$ is $\mathbb{Q}$-Cartier.
Let $\beta\colon Z \to X$ be a projective birational morphism. 
Then there are uniquely defined rational numbers $a(E,X,\sG)$ such that
$$
K_{\beta^{-1}{\sG}}\sim_\mathbb{Q}\beta^*K_{\sG}+ \sum_E a(E,X,\sG)E,
$$
where $E$ runs through all exceptional prime divisors for $\beta$.
The rational numbers $a(E,X,\sG)$ do not depend on the birational morphism $\beta$,
but only on the valuations associated to the $E$. 
We say that $\sG$ is \textit{canonical} if, for all $E$ exceptional over $X$, $a(E,X,\sG) \ge 0$.
\end{defn}

We finally recall the behaviour of canonical singularities with respect to birational maps and finite covers.

\begin{lemma}[{\cite[Lemma 4.2]{cd1fzerocan}}]\label{lemma:singularities_birational_morphism}
Let $\beta\colon Z \to X$ be a birational projective morphism of normal complex varieties, and let $\sG$ be a foliation on $X$.
Suppose that $K_\sG$ is $\mathbb{Q}$-Cartier.
\begin{enumerate}
\item Suppose that $K_{\beta^{-1}\sG}\sim_\mathbb{Q} \beta^*K_\sG+E$ for some effective $\beta$-exceptional 
$\mathbb{Q}$-divisor on $Z$. If 
$\beta^{-1}\sG$ is canonical, then so is $\sG$.
\item If $K_{\beta^{-1}\sG}\sim_\mathbb{Q} \beta^*K_\sG$, then $\sG$ is canonical if and only if so is $\beta^{-1}\sG$.
\end{enumerate}
\end{lemma}

\begin{lemma}[{\cite[Lemma 4.3]{cd1fzerocan}}]\label{lemma:canonical_quasi_etale_cover}
Let $f\colon X_1 \to X$ be a quasi-finite dominant morphism of normal complex varieties, and let $\sG$ be a foliation on $X$ with
$K_\sG$ $\mathbb{Q}$-Cartier. Suppose that any codimension one component of the branch locus of $f$ is 
$\sG$-invariant. If $\sG$ is canonical, then so is $f^{-1}\sG$.
\end{lemma}

\subsection{Bertini-type results}

The present paragraph is devoted to the following auxiliary results.

\begin{lemma}\label{lemma:bertini_2}
Let $\beta\colon X \to Y$ be a birational morphism of smooth quasi-projective varieties with $n:=\dim X \ge 3$.
Let $\sG$ be a codimension one foliation on $X$ given by a twisted $1$-form $\omega\in H^0(X,\Omega_X^1\otimes\sL)$. Suppose that any $\beta$-exceptional prime divisor on $X$ dominates a codimension $2$ closed subset in $X$. 
If $A$ is a general member of a very ample linear system $|A|$ on $Y$, then the zero set of the induced twisted $1$-form $\omega_{H}\in H^0(H,\Omega_H^1\otimes\sL_{|H})$ on $H:=\beta^{-1}(A)$ has codimension at least two. 
\end{lemma}

\begin{proof}Set $|H|:=\beta^*|A|$. Recall that there exists a composition $Z \to Y$ of a finite number of blow-ups with smooth centers such that the induced rational map $Z \map X$ is a morphism. This immediately implies that $d\beta$ has rank $n-1$ at the generic point of any $\beta$-exceptional divisor. 
Let $X^\circ$ be an open set in $X$ with complement of codimension at least $2$ such that $\beta^{-1}\sG$ is regular on $X^\circ$ and such that $d\beta_x(\sG_x)$ has rank at least $n-2$ at any point on some $\beta$-exceptional divisor in $X^\circ$.
Consider 
$$
I^\circ=\left\{(x,H)\in X^\circ \times |H|\text{ such that } x\in H\text{ and }\sG_x \subseteq T_xH\right\}.
$$
We denote by $p\colon I^\circ \to X^\circ$ the projection. If $x \in X^\circ \setminus \textup{Exc}(\beta)$, then $p^{-1}(x) \subset |H|$ is a linear subspace of dimension $\dim |H|-n$.
If $x \in \textup{Exc}(\beta)$, then $p^{-1}(x) \subset |H|$ is a linear subspace of dimension at most $\dim |H|-(n-1)$ since $d\beta_x(\sG_x)$ has rank at least $n-2$ by choice of $X^\circ$.
It follows that any irreducible component of $I^\circ$ has dimension at most $\dim |H|$. Thus general fibers of the second projection 
$q \colon I^\circ \to |H|$ have dimension $\le 0$. Our claim then follows easily.
\end{proof}

\begin{lemma}\label{lemma:bertini}
Let $\psi\colon X \to Y$ be a dominant and equidimensional morphism of smooth quasi-projective varieties with reduced fibers.
Suppose in addition that $\dim Y = \dim X -1 \ge 2$. Let $\sG$ be a codimension one foliation on $X$ given by a twisted $1$-form $\omega\in H^0(X,\Omega_X^1\otimes\sL)$. Suppose that the generic fiber of $\psi$ is not tangent to $\sG$.
If $A$ is a general member of a very ample linear system on $Y$, then the zero set of the induced twisted $1$-form $\omega_{H}\in H^0(H,\Omega_H^1\otimes\sL_{|H})$ 
on $H:=\psi^{-1}(A)$ has codimension at least two.
\end{lemma}

\begin{proof}
This also follows from an easy dimension count (see proof of Lemma \ref{lemma:bertini_2}).
\end{proof}

The proof of Proposition \ref{prop:bertini} below makes use of the following result of McQuillan.

\begin{prop}[{\cite[Facts I.1.8 and I.1.9]{mcquillan08}}]\label{prop:surface_singularity}
Let $\sL$ be a foliation of rank one on a smooth complex quasi-projective surface, and let $x$ be a singular point of $\sL$. Let also $v$ be a local generator of $\sL$ in a neighbourhood of $x$. Then $\sL$ is not canonical at $x$ if and only if the linear part $D_xv$ of $v$ at $x$ is nilpotent or diagonalizable with nonzero eigenvalues $\lambda$ and $\mu$ satisfying in addition $\frac{\lambda}{\mu}\in \mathbb{Q}_{>0}$. 
\end{prop}

\begin{fact}\label{fact:nilpotent} Notation as in Proposition \ref{prop:surface_singularity}.
If $D_xv$ is nilpotent, then the following holds (see proof of \cite[Theorem 1.1]{brunella}). 
The exceptional divisor $E_1$ of the blow-up $S_1$ of $S$ at $x$ has discrepancy $a(E_1) \le 0$.

Suppose that $a(E_1) = 0$. Then the induced foliation $\sL_1$ on $S_1$ has a unique singular point $x_1$ on $E_1$. Moreover, if $v_1$ is a local generator of $\sL_1$ in a neighbourhood of $x_1$, then $D_{x_1}v_1$ is nilpotent as well, and the exceptional divisor $E_2$ of the blow-up $S_2$ of $S_1$ at $x_1$ has discrepancy $a(E_2)\le 0$.

If $a(E_2)=0$, then the induced foliation $\sL_2$ on $S_2$ has a unique singular point $x_2$ on $E_2$, and $D_{x_2}v_2$ is zero, where $v_2$ denotes a local generator of $\sL_2$ in a neighbourhood of $x_2$. Moreover, the exceptional divisor $E_3$ of the blow-up $S_2$ of $S_1$ at $x_2$ has discrepancy $a(E_3) \le -1$.
\end{fact}

\begin{fact}\label{fact:diagonalizable} Notation as in Proposition \ref{prop:surface_singularity}.
Suppose that $D_xv$ is diagonalizable and that its eigenvalues $\lambda$ and $\mu$ are nonzero and satisfy 
$\frac{\lambda}{\mu}=:r\in \mathbb{Q}_{>0}$. Let $S_1$ be the blow-up of $S$ at $x$ with exceptional divisor $E_1$, and let $\sL_1$ be the foliation on $S_1$ induced by $\sL$. 

Suppose that $\lambda\neq \mu$.
A straightforward local computation then shows that $v$ extends to a regular vector field $v_1$ on some open neighbourhood of $E_1$ with isolated zeroes. Moreover, $\sL_1$ has two singularities with diagonalizable linear parts and the quotients of the eigenvalues of the linear parts are 
$r-1$ and $\frac{1}{r}-1$. The divisor $E_1$ has discrepancy $a(E_1)=0$.

If $\lambda= \mu$, then $E_1$ has discrepancy $a(E_1)\le -1$. 

In either case, the Euclidean algorithm implies that there exists a divisorial valuation with center $x$ and negative discrepancy.
\end{fact}

\begin{prop}\label{prop:bertini}
Let $X$ be a smooth quasi-projective variety and let $\sG$ be a codimension $1$ foliation on $X$ with canonical singularities. 
Let also $B$ be a smooth codimension two component of the singular set of $\sG$.
Let $S \subseteq X$ be a two dimensional complete intersection of general elements of a very ample linear system $|H|$ on $X$, and let $\sL$ be the foliation of rank $1$ on $S$ induced by $\sG$. Then $\sL$ has canonical singularities in a Zariski open neighbourhood of $B \cap S$.
\end{prop}

\begin{proof} Set $n:=\dim X$. Suppose that $\dim X \ge 3$. Let $U \subseteq |H|^{n-2}$ be a dense open set such that $S_u:=X\cap H_1\cap \cdots \cap H_{n-2}$ is a smooth connected surface for any $u=(H_1,\ldots,H_{n-2}) \in U$, and let $\sL_u$ be the foliation of rank one on $S_u$ induced by $\sG$.
Let $T:=\{(x,u)\in B \times U\,|\, x \in B \cap S_u\}$, and denote by $p \colon T \to U$ and $q \colon T \to B$ the natural morphisms.
Shrinking $U$, if necessary, we may assume that $p$ is a finite \'etale cover. Let $\omega\in H^0(X,\Omega_X^1\otimes\sN)$ be a 
twisted $1$-form defining $\sG$. By Lemma \ref{lemma:bertini_2}, we may also assume that, for any $t=(x,u)\in T$,  
the induced twisted $1$-form $\omega_u\in H^0(S_u,\Omega_{S_u}^1\otimes\sN_{|S_u})$ 
on $S_u$ has isolated zeroes.
Given $t=(x,u)\in T$, let $v_t$ be a local generator of $\sL_u$ in a neighourhood of $x$. Note that $x$ is a singular point of $\sL_u$ since $\omega_u$ vanishes at $x$ and has isolated zeroes.

\medskip

In order to prove the proposition, we argue by contradiction and assume that the set of points $t=(x,u) \in T$ such that $\sL_u$ is not canonical at $x$ is dense in $T$. 

\medskip

Suppose that $D_xv_t$ is diagonalizable with nonzero eigenvalues $\lambda_t$ and $\mu_t$ satisfying $\frac{\lambda_t}{\mu_t}\in\mathbb{Q}_{>0}$ for some $t=(x,u) \in T$. Let $\Omega$ be a holomorphic $1$-form defining $\sG$ in some analytic open neighbourhood $W$ of $x$. Then we must have $d_x\Omega \neq 0$ since the restriction of $\Omega$ to $W \cap S_u$ defines ${\sL_u}_{|W}$ by choice of $U$.
A theorem of Kupka (see \cite{kupka}) then says that, shrinking $W$, if necessary, the $2$-form $d\Omega$ on $W$ defines a codimension two foliation tangent to $\sG_{|W}$. Therefore, shrinking $W$ further, there exists analytic coordinates $(x_1,\ldots,x_n)$ centered at $x$ on $W$ such that $\sG_{|W}$ is defined by the $1$-form $a(x_1,x_2)dx_1+b(x_1,x_2)dx_2$ for some holomorphic function $a$ and $b$ defined in a neighbourhood of $0$ in $\mathbb{C}^2$.

On the other hand, given $r \in \mathbb{Q}_{>0}$, the set of points $t=(x,u) \in T$ such that $D_xv_t$ is diagonalizable with eigenvalues $\lambda_t \neq 0$ and $r\lambda_t \neq 0$ is locally closed for the Zariski topology. Moreover,
the set of points $t=(x,u) \in T$ such that the linear part $D_xv_t$ of $v_t$ at $x$ is nilpotent is Zariski closed.  

Therefore, shrinking $U$ again, if necessary, we may assume that one of the following holds.
\begin{enumerate}
\item The linear part $D_xv_t$ is nilpotent for any $t=(x,u) \in T$.
\item There exists $r\in\mathbb{Q}_{>0}$ such that 
$D_xv_t$ is diagonalizable with nonzero eigenvalues $\lambda_t$ and $r\lambda_t$ for any $t=(x,u) \in T$.
\end{enumerate}

\medskip

\noindent\textit{Case 1.} Suppose first that $D_xv_t$ is nilpotent for any $t=(x,u) \in T$. Let $\beta_1\colon X_1 \to X$ be the blow-up of $X$ along $B$ with exceptional divisor $E_1$, and let $\sG_1$ be the foliation on $X_1$ induced by $\sG$. Notice that we have
\begin{equation}\label{eq:can_1}
K_{\sL_u} \sim_\mathbb{Z} \big(K_\sG+H_1+\cdots+H_{n-2}\big)_{|S_u}.
\end{equation}
Set $S_{1,u}:=\beta^{-1}(H_1)\cap\cdots\cap \beta^{-1}(H_{n-2})$ and denote by $\sL_{1,u}$ the foliation on $S_{1,u}$ induced by $\sL_u$.
Notice that $S_{1,u}$ is the blow-up of $S_u$ along $B\cap S_u$. 
By Lemma \ref{lemma:bertini_2} above, we also have
\begin{equation}\label{eq:can_2}
K_{\sL_{1,u}} \sim_\mathbb{Z} \big(K_{\sG_1}+\beta^*H_1+\cdots+\beta^*H_{n-2}\big)_{|S_{1,u}}.
\end{equation}
Let $t=(x,u) \in T$ and set $E_{1,t}:=\beta^{-1}(x)\subset S_{1,u}$. From equations \eqref{eq:can_1} and \eqref{eq:can_2}, we conclude that
\begin{equation*}
a(E_{1,t},S_{1,u},\sL_{1,u})=a(E,X,\sG).
\end{equation*} 
By Fact \ref{fact:nilpotent}, we have $a(E_{1,t},S_{1,u},\sL_{1,u}) \le 0$. On the other hand, by assumption, we must have
$a(E,X,\sG)\ge 0$. It follows that 
$$a(E_{1,t},S_{1,u},\sL_{1,u})=a(E,X,\sG)=0$$ 
for any $t=(x,u) \in T$. Moreover, $\sL_{1,u}$ has a unique singular point $x_{1,t}$ on $E_{1,t}$ by Fact \ref{fact:nilpotent} again. 

We claim that there exists a codimension two irreducible component $B_1 \subset E_1$ of the singular set of $\sG_1$ dominating $B$.
Suppose otherwise. Then $\sG_1$ is regular along a general fiber $\ell$ of the projection $E_1 \to B$. We have $K_{\sG_1}\cdot \ell =0$ since 
$K_{\sG_1}=\beta^*K_\sG$. On the other hand, we have $K_{X_1}\cdot \ell =-1$ by construction. It follows that
that $\sN_{\sG_1}\cdot \ell=1$. This immediately implies that $\ell$ is tangent to $\sG_1$ since 
$\deg \Omega_\ell^1\otimes {\sN_{\sG_1}}_{|\ell}=-1$. But this contradicts
Lemma \ref{lemma:residue_orbifold} below. Note that $B_1$ is unique since $\sL_{1,u}$ has a unique singular point $x_{1,t}$ on $E_{1,t}$ and $B_1 \cap E_{1,t}$ is contained in the singular set of $\sL_{1,u}$.

Let $\beta_2\colon X_2 \to X_1$ be the blow-up of $X_1$ along $B_1$ with exceptional divisor $E_2$, and let $\sG_2$ be the foliation on $X_2$ induced by $\sG$. Arguing as above, we see that we must have $a(E_2,X,\sG)=0$. Moreover, there exists a unique codimension two irreducible component $B_2 \subset E_2$ of the singular set of $\sG_2$ dominating $B$.

Let now $E_3$ be the exceptional divisor of the blow-up of $X_2$ along $B_2$. Arguing as above, we see that must have $a(E_2,X,\sG) \le -1$ by Step 1, yielding a contradiction.

\medskip

\noindent\textit{Case 2.} Suppose now that there exists $r\in\mathbb{Q}_{>0}$ such that 
$D_xv_t$ is diagonalizable with eigenvalues $\lambda_t$ and $r\lambda_t$ for any $t=(x,u) \in T$.
Arguing as in Case 1, one shows that there exists a divisorial valuation with center $B$ on $X$ and negative discrepancy. One only needs to replace the use of Fact \ref{fact:nilpotent} by Fact \ref{fact:diagonalizable}.
This yields again a contradiction, completing the proof of the proposition.
\end{proof}

\section{Basic results}

In this section, we extend a number of earlier results to the context of normal quasi-projective varieties with quotient singularities. See \cite{correa_bott_residue} for a somewhat related result.

\subsection{Algebraic and analytic $\mathbb{Q}$-structures} We will use the notions of $\mathbb{Q}$-varieties and $\mathbb{Q}$-sheaves without further explanation and simply refer to \cite[Section 2]{mumford_enumerative_geometry}.

\begin{notation}\label{Q-varieties} Let $X$ be a normal quasi-projective variety, and let
$X_\mathbb{Q}:= \big(X,\{p_{\alpha}:X_{\alpha}\to X\}_{\alpha \in A}\big)$ be a structure of $\mathbb{Q}$-variety on $X$. For each $\alpha\in A$, 
$X_\alpha$ is smooth and quasi-projective, and that there exists a normal variety $U_\alpha$ and a factorization of $p_\alpha$ 
\begin{center}
\begin{tikzcd}
X_\alpha \ar[rrrrrr, bend left=15, "p_\alpha"]\ar[rrrr, "{q'_\alpha,\textup{ Galois with group } G_\alpha}"'] &&&& U_{\alpha} \ar[rr, "{p'_\alpha,\textup{ \'etale}}"'] && X.
\end{tikzcd}
\end{center}
We will denote by $X_{\alpha\beta}$ the normalization of $X_\alpha\times_X X_\beta$, and by 
$p_{\alpha\beta,\alpha}\colon X_{\alpha\beta}\to X_\alpha$ and $p_{\alpha\beta,\beta}\colon X_{\alpha\beta} \to X_\beta$ the natural morphisms. 
Both $p_{\alpha\beta,\alpha}$ and $p_{\alpha\beta,\beta}$ are \'etale by the very definition of a $\mathbb{Q}$-structure.

In $\textit{loc. cit.}$, Mumford constructs a \textit{global cover} of $X_\mathbb{Q}$, that is, a quasi-projective normal variety $\wh{X}$, a finite Galois cover
$p:\widehat X\to X$ with group $G$, and for every $\alpha\in A$, a commutative diagram as follows:

\begin{center}
\begin{tikzcd}
\wh{X}_{\alpha} \ar[rrrrr, bend left=15, "p_{|\wh{X}_{\alpha}}"]\ar[rrr, "{q_{\alpha},\textup{ Galois cover}}"'] &&& X_{\alpha} \ar[rr, "p_{\alpha}"'] && X,
\end{tikzcd}
\end{center}
where $\wh{X}_\alpha\subseteq \wh{X}$ is an open. The finite map $(q_\alpha,q_\beta)\colon \wh X_{\alpha\beta}:=\wh X_\alpha \cap \wh X_\beta \to X_\alpha\times_X X_\beta$ factors through the normalization map $X_{\alpha\beta}\to X_\alpha\times_X X_\beta$ and induces a finite morphism $q_{\alpha\beta}\colon \wh X_{\alpha\beta} \to X_{\alpha\beta}$.
\end{notation}

\begin{fact}\label{fact:quotient_singularities_Q_strucutures}
Let $X$ be a normal quasi-projective variety with only quotient singularities. Then according to Mumford (\cite[Section 2]{mumford_enumerative_geometry}) there exists a strucutre of $\mathbb{Q}$-variety on $X$ given by a collection of charts $\{p_{\alpha}\colon X_{\alpha}\to X\}_{\alpha \in A}$ where $p_{\alpha}$ is \'etale in codimension one for every $\alpha\in A$. We will refer to it as a \textit{quasi-\'etale $\mathbb{Q}$-variety structure}.
\end{fact}

\begin{fact}\label{fact:purity}
Let $X$ be a normal quasi-projective variety, and let $\{p_\alpha \colon X_\alpha\to X\}_{\alpha \in A}$ be a finite set of morphisms
such that $X=\cup_{\alpha\in A}p_\alpha(X_\alpha)$. Suppose that we have a factorization of $p_\alpha = p'_\alpha\circ q'_\alpha$ as above. Then 
the collection of morphisms $\{p_\alpha \colon X_\alpha\to X\}_{\alpha \in A}$ automatically defines a quasi-\'etale $\mathbb{Q}$-variety structure on $X$ by purity of the branch locus.
\end{fact}

In the setting of Fact \ref{fact:quotient_singularities_Q_strucutures}, there is a finite covering $(V_i)_{i\in I}$ of $X$ by analytically open sets such that the following holds. For  
each $i\in I$, there exists $\alpha(i) \in A$ such that $V_i \subseteq p_{\alpha(i)}(X_{\alpha(i)})$ and a connected component 
$V'_i$ of $(p'_{\alpha(i)})^{-1}(V_i)$ such that the restriction of $p'_{\alpha(i)}$ to $V'_i$ induces an isomorphism onto $V_i$.
Set $X_i:=(q'_{\alpha(i)})^{-1}(V'_i)$ and $\wh X_i:=q_{\alpha(i)}^{-1}(X_i)$. We have a commutative diagram as follows:
\begin{center}
\begin{tikzcd}[column sep=large, row sep=large]
\wh X_i  \ar[r, hookrightarrow]\ar[d, "{q_i,\textup{ finite}}"'] & {\wh X}_{\alpha(i)} 
\ar[r, hookrightarrow]\ar[d, "{q_{\alpha(i)}}"] & \wh X\ar[dd, "{p}"]\\
X_i \ar[r, hookrightarrow]\ar[d, "{q'_i,\textup{ quasi-\'etale and Galois with group } G_i}"']\ar[dr, "{p_i}"] & X_{\alpha(i)} \ar[d, "{p_{\alpha(i)}}"] & \\
V_i \ar[r, hookrightarrow] & X \ar[r, equal] & X.
\end{tikzcd}
\end{center}
Note that we have $X=\cup_i p_i(X_i)$ by construction.
Let $X_{ij}$ denotes the normalization of $X_i\times_X X_j$. The natural morphisms
$p_{ij,i}\colon X_{ij}\to X_i$ and $p_{ij,j}\colon X_{ij} \to X_j$ are \'etale by purity of the branch locus and $X_{ij}$ is smooth.
The finite map $X_{ij} \to V_i \cap V_j$ is Galois with group $G_i \times G_j$. 
Moreover, the finite map $(q_i,q_j)\colon \wh X_{ij}:=\wh X_i \cap \wh X_j \to X_i\times_X X_j$ induces a finite morphism $q_{ij}\colon \wh X_{ij} \to X_{ij}$.
We view $p \colon \wh{X} \to X$ as a global cover for the analytic quasi-\'etale $\mathbb{Q}$-structure given by the collection of charts $\{p_{i}\colon X_{i}\to X\}_{i \in I}$. 

Notice that the collections of open sets $(g\cdot\wh{X}_i)_{g\in G,i\in I}$ and $(g\cdot \wh{X}_\alpha)_{g\in G,\alpha \in A}$ both form a covering of $\wh{X}$.

\subsection{A basic formula} In the present paragraph, we extend \cite[Proposition 2.2]{brunella} to surfaces with quotient singularities.

\begin{lemma}\label{lemma:log_canonical_invariant_curve}
Let $X$ be a normal quasi-projective surface with quotient singularities, and let $\sG \subset T_X$ be a foliation of rank one. Let $C$ be an irreducible complete curve on $X$ and suppose that $C$ is transverse to $\sG$ at a general point on $C$. Then 
$K_\sG\cdot C+C\cdot C \ge 0$.
\end{lemma}

\begin{rem}
Quotient singularities are $\mathbb{Q}$-factorial so that $K_\sG$ and $C$ are $\mathbb{Q}$-Cartier divisors.
\end{rem}

\begin{proof}[{Proof of Lemma \ref{lemma:log_canonical_invariant_curve}}]
The proof is very similar to that of \cite[Proposition 2.2]{brunella} and so we leave some easy details to the reader.

\medskip

Let $m$ be a positive integer such that $mK_\sG$ and $mC$ are Cartier divisors.
Let $X_\mathbb{Q}=\big(X,\{p_{\alpha}:X_{\alpha}\to X\}_{\alpha\in A}\big)$ be a quasi-\'etale $\mathbb{Q}$-structure on $X$ (see Fact \ref{fact:quotient_singularities_Q_strucutures}). We will use the notation of paragraph \ref{Q-varieties}.
Let $B \to X$ be the normalization of $C$ and let $\wh B$ be the normalization of some irreducible component of $p^{-1}(C)$. We have a commutative diagram as follows:
\begin{center}
\begin{tikzcd}
\wh{B} \ar[r]\ar[d] & \wh{X} \ar[d, "p"] \\
B \ar[r] & X.
\end{tikzcd}
\end{center}
Shrinking the $X_\alpha$, if necessary, we may assume that for each $\alpha\in A$ there exists a generator 
$v_\alpha$ of $p_\alpha^{-1}\sG$ on $X_\alpha$. We may also assume without loss of generality that there exists
a nowhere vanishing regular function $h_\alpha$ on $X_\alpha$ such that  
$h_\alpha v_\alpha^{\otimes m}$ is $G_\alpha$-invariant since $mK_\sG$ is a Cartier divisor by choice of $m$. Replacing $X_\alpha$ by an \'etale cover, if necessary, we can suppose that $h_\alpha^{\frac{1}{m}}$ is regular on $X_\alpha$, and replacing $v_\alpha$ by $h_\alpha^{\frac{1}{m}}v_\alpha$ we may finally assume that $v_\alpha^{\otimes m}$ is $G_\alpha$-invariant.
Similarly, we can suppose that the ideal sheaf $p_\alpha^{-1}\sI_C$ is generated by a regular function $f_\alpha$ such that 
$f_\alpha^m$ is $G_\alpha$-invariant. We will denote by $p_\alpha^{-1}C$ the scheme defined by equation $\{f_\alpha = 0\}$. Notice that the corresponding $\mathbb{Q}$-sheaves $\sG_\mathbb{Q}$ and $\sO_X(-C)_\mathbb{Q}$ are $\mathbb{Q}$-line bundles, and the pull-back $\sM$ of $\sG_\mathbb{Q}^*\otimes\sO_X(C)_\mathbb{Q}$ to $\wh B$ is a genuine line bundle.

One then readily checks that the functions $(v_\alpha(f_\alpha))^m$ restricted to $p_\alpha^{-1}C$ give a non-zero global section of $\sM^{\otimes m}$. In particular, $\sM$ has non negative degree. Observe that the pull-back of $\sG_\mathbb{Q}^{\otimes m}$ (resp. $\sO_X(C)_\mathbb{Q}^{\otimes m}\cong \sO_X(mC)_\mathbb{Q}$) to $\wh{X}$ is isomorphic to $p^*\sO_X(-mK_\sG)$
(resp. $p^*\sO_X(mC)$) so that $\sM^{\otimes m}$ is isomorphic to the pull-back of $\sO_X(mK_\sG+mC)$ to $\wh B$. The lemma then follows from the projection formula.
\end{proof}

\subsection{The Baum-Bott partial connection} The following result generalizes \cite[Corollary 3.4]{baum_bott70} to the setting of quasi-projective vareities with quotient singularities. 

\begin{lemma}\label{lemma:residue_orbifold}
Let $X$ be a normal quasi-projective variety with quotient singularities, and let $\sG \subset T_X$ be a codimension one foliation on $X$. 
Let also $\{p_i:X_i\to X\}_{i \in I}$ be an analytic quasi-\'etale $\mathbb{Q}$-structure on $X$ and suppose that 
$p_i^{-1}\sG$ is defined by a $1$-form $\omega_i$ with zero set of codimension at least two such that $d\omega_i=\alpha_i\wedge \omega_i$ for some holomorphic $1$-form $\alpha_i$ on $X_i$.
Then the following holds.
\begin{enumerate}
\item We have $c_1(\sN_\sG)^{2}\equiv 0$.
\item Let $Y \subseteq X$ be a projective subvariety. Suppose that $Y$ is not entirely contained in the union of the singular loci of $X$ and $\sG$.
Suppose moreover that $Y$ is tangent to $\sG$. Then ${c_1(\sN_\sG)}_{|Y} \equiv 0$.
\end{enumerate}
\end{lemma}

\begin{proof} We use the notation of paragraph \ref{Q-varieties}. Set $\sL:=\sN^*_\sG$ and $\wh{\sL}:=p^{[*]}\sL$. By \cite[Proposition 1.9]{hartshorne80}, the reflexive pull-back $p_\alpha^{[*]}\sL$ is a line bundle. It follows that ${\wh \sL}_{\,\,|\wh X_\alpha}\cong q_\alpha^*\big(p_\alpha^{[*]}\sL\big)$ since both sheaves are reflexive and agree over the big open set $X_{\textup{reg}}$. This shows that $\wh \sL$ is a line bundle. Moreover,
if $m$ is a positive integer such that $\sL^{[\otimes m]}$ is a line bundle, then we have
$p^*(\sL^{[\otimes m]})\cong {\wh\sL}^{\,\,\otimes m}$.

Let $\omega\in H^0\big(X,\Omega^{[1]}_X\boxtimes \sN_\sG\big)$ be a twisted $1$-form defining $\sG$. The reflexive pull-back of $\omega$ then gives an inclusion ${\wh{\sL}} \subset p^{[*]}\Omega_X^{[1]}$, which is saturated by \cite[Lemma 9.7]{fano_fols}. Note that we have $p^{[*]}\Omega_X^{[1]}\subseteq \Omega_{\wh X}^{1}\subseteq \Omega_{\wh X}^{[1]}$ since we have a factorization $p_{|\wh{X}_\alpha}=p_\alpha \circ q_\alpha$ and $X_\alpha$ is smooth.

Shrinking the $V_i$, is necessary, we may assume that there exist nowhere vanishing holomorphic functions $h_i$ on $X_i$ such that 
$h_i\omega_i^{\otimes m}$ is $G_i$-invariant. Observe that  
$d\big(h_i^{\frac{1}{m}}\omega_i\big)=\big(\frac{1}{m}\frac{dh_i}{h_i}+\alpha_i\big)\wedge \big(h_i^{\frac{1}{m}}\omega_i\big)$
so that, replacing $\omega_i$ by $h_i^{\frac{1}{m}}\omega_i$ and $\alpha_i$ by $\frac{1}{m}\frac{dh_i}{h_i}+\alpha_i$, we can suppose that $\omega_i^{\otimes m}$ is $G_i$-invariant.

Now, we can write $${\omega_j}_{|X_{ij}} = \phi_{ij}\, {\omega_i}_{|X_{ij}}$$ on 
$X_{ij}$ where $\phi_{ij}$ is a nowhere vanishing holomorphic function since the
$1$-forms ${\omega_i}_{|X_{ij}}$ and ${\omega_j}_{|X_{ij}}$ both define the foliation $p_{ij,i}^{-1}p_i^{-1}\sG=p_{ij,j}^{-1}p_j^{-1}\sG$ on $X_{ij}$ and their zero sets set have codimension at least $2$ in $X_{ij}$. 
The holomorphic functions $(\phi_{ij}\circ q_{ij})^m$ on $\wh{X}_{ij}$ then give a cocycle with respect to the open covering
$(g\cdot \wh{X}_i)_{g \in G, i \in I}$ that corresponds to the isomorphism class of $\wh \sL^{\,\,\otimes m}$ as a complex analytic line bundle since it does over the open set $X_\textup{reg}$.

Let $c\in H^1\big(\wh X,p^{[*]}\Omega_X^{[1]}\big)$ be
the cohomology class corresponding to the cocyle 
with respect to the open covering $(g\cdot \wh{X}_i)_{g \in G, i \in I}$
induced by the $q_{ij}^*\frac{d\phi_{ij}}{\phi_{ij}}$.
Since $d\omega_i=\alpha_i\wedge \omega_i$ for any $i\in I$, we must have 
$$\Big(\frac{d\phi_{ij}}{\phi_{ij}}+{\alpha_i}_{|X_{ij}}-{\alpha_j}_{|X_{ij}}\Big)\wedge {\omega_i}_{|X_{ij}}=0.$$ 
This immediately implies that $q_{ij}^*\big(\frac{d\phi_{ij}}{\phi_{ij}}+{\alpha_i}_{|X_{ij}}-{\alpha_j}_{|X_{ij}}\Big) \in H^0\big(\wh X_{ij},{\wh{\sL}}_{\,\,|\wh X_{ij}}\big) \subseteq  
H^0\big(\wh X_{ij}, {p^{[*]}\Omega_X^{[1]}}_{|\wh X_{ij}}\big)$
since ${\wh{\sL}}$ is saturated in $p^{[*]}\Omega_X^{[1]}$, and shows that
$c$ is the image of a cohomological class $b \in H^1\big(\wh X,{\wh\sL}\,\,\big)$ under
the natural map $H^1\big(\wh X,{\wh\sL}\,\,\big) \to H^1\big(\wh X,p^{[*]}\Omega_X^{[1]}\big).$
By construction, $c \in H^1\big(\wh X,p^{[*]}\Omega_X^{[1]}\big)$ maps to
$c_1\big(\wh{\sL}\,\,\big)\in H^1\big(\wh X,\Omega_{\wh X}^{1}\big)$
under the natural map
$H^1\big(\wh X,p^{[*]}\Omega_X^{[1]}\big) \to H^1\big(\wh X,\Omega_{\wh{X}}^{1}\big)$.
On the other hand, $b \cup b =0 \in H^1\big(\wh X,\wedge^2\wh{\sL}\,\,\big)$ since $\wh{\sL}$ is a line bundle.
This immediately implies that $c_1\big(\wh{\sL}\,\,\big)^2 = 0 \in H^2\big(\wh X,\Omega_{\wh X}^{2}\big)$, proving (1).

Let $Y \subseteq X$ be a projective subvariety, and let $\wh Y$ be a resolution of some irreducible component of $p^{-1}(Y)$. We have a commutative diagram as follows:
\begin{center}
\begin{tikzcd}
\wh Y \ar[r, "f"]\ar[d, "{\textup{ generically finite}}"'] & \wh X\ar[d, "{p, \textup{ finite}}"]\\
Y \ar[r, hookrightarrow] & X.
\end{tikzcd}
\end{center}
Note that $c \in H^1\big(\wh X,p^{[*]}\Omega_X^{[1]}\big)$ maps to $c_1\big(f^*\wh{\sL}\,\,\big)\in H^1\big(\wh Y,\Omega_{\wh Y}^{1}\big)$
under the composed map
\begin{center}
\begin{tikzcd}[column sep=huge]
H^1\big(\wh X,p^{[*]}\Omega_X^{[1]}\big) \ar[r] & H^1\big(\wh X,\Omega_{\wh X}^{1}\big) 
\ar[r] & H^1\big(\wh Y,\Omega_{\wh Y}^{1}\big).
\end{tikzcd}
\end{center}
Suppose moreover that $Y$ is not entirely contained in the union of the singular loci of $X$ or $\sG$, and that 
it is tangent to $\sG$. Then the composed map of sheaves 
\begin{center}
\begin{tikzcd}
f^*{\wh\sL} \ar[r] & f^*\big(p^{[*]}\Omega_X^{[1]}\big) \ar[r] &
f^*\Omega_{\wh X}^{1} \ar[r] & \Omega_{\wh Y}^{1}
\end{tikzcd}
\end{center}
vanishes. This immediately implies that $c_1\big(f^*\wh{\sL}\,\,\big) = 0 \in H^1\big(\wh Y,\Omega_{\wh Y}^{1}\big)$. Item (2) now follows from the projection formula.
\end{proof}

\subsection{Baum-Bott formula} The next result extends the Baum-Bott formula to surfaces with quotient singularities.

\begin{defn}Let $X$ be a normal quasi-projective algebraic surface with quotient singularities, and let $\sL \subset T_X$ be a foliation of rank one.
Given $x \in X$, there exist an open analytic neighbourhood $U$ of $x$, a (not necessarily connected) smooth analytic complex manifold $V$
and a finite Galois holomorphic map $p\colon  V\to U$ that is \'etale outside of the singular locus. Set
$$\textup{BB}^\mathbb{Q}(\sL,x):=\frac{1}{\deg p}\sum_{y \in p^{-1}(x)} \textup{BB}(p^{-1}\sL_{|U},y)$$ 
where $\textup{BB}(p^{-1}\sL_{|U},y)$ denotes the Baum-Bott index of 
$p^{-1}\sL_{|U}$ at $y$ (we refer to \cite[Section 3.1]{brunella} for this notion).
\end{defn}

\begin{rem}
One readily checks that $\textup{BB}^\mathbb{Q}(\sL,x)$ is independent of the local chart $p \colon V \to U$ at $x$.
\end{rem}

\begin{prop}\label{prop:BB_orbifold}
Let $X$ be a normal projective algebraic surface with quotient singularities, and let $\sL \subset T_X$ be a foliation of rank one.
Then $$c_1(\sN_\sL)^2 =\sum_{x\in X} \textup{BB}^\mathbb{Q}(\sL,x).$$
\end{prop}

\begin{proof}
The proof is very similar to that of \cite[Theorem 3.1]{brunella} and so we leave some easy details to the reader.
We use the notation of paragraph \ref{Q-varieties}. Let $m$ be a positive integer such that $\sN_\sL^{[\otimes m]}$ is a line bundle and set 
$\wh{\sN}_\sL:=p^{[*]}\sN_\sL$.
Shrinking the $V_i$, is necessary, we may assume without loss of generality that
${p_i}^{-1}\sL$ is defined by a $1$-forms $\omega_i$ with isolated zeroes, and that
there exist a smooth $(1,0)$-form $\beta_i$ on $X_i$ and a small enough open set $W_i\Subset V_i$ such that $d\omega_i=\beta_i \wedge \omega_i$ on $X_i \setminus p_i^{-1}(\wb{W}_i)$. We may assume that $W_i$ is the image in $V_i$ of some small enough open ball in $X_i$.
We may also assume that $\omega_i$ is nowhere vanishing on $X_i$ or that it vanishes at a single point contained in $p_i^{-1}(W_i)$,
and that $W_i\cap V_j=\emptyset$ if $i \neq j$. 
One then observes that
$$\textup{BB}({p_i}^{-1}\sL,x) = \frac{1}{(2i\pi)^2}\int_\Sigma\beta\wedge d\beta$$
for any $x \in X_i$, where $\Sigma$ is a small enough suitably oriented $3$-sphere in $X_i$ centered at $x$ and $\beta$ is any smooth $(1,0)$-form on $X_i$ such that $d\omega_i=\beta\wedge\omega_i$ in a neighbourhood of $\Sigma$. 
Finally, we can suppose that $\omega_i^{\otimes m}$ is $G_i$-invariant (see proof of Lemma \ref{lemma:residue_orbifold}) and that $\beta_i$ vanishes identically on some open neighbourhood of the singular locus of $\omega_i$ and some open neighbourhood of the inverse image of the singular set of $X$ in $X_i$. Since $\omega_i$ is semi-invariant under $G_i$, replacing $\beta_i$ by 
$\frac{1}{\sharp\, G_i}\sum_{g\in G_i}g^*\beta_i$, if necessary, we may assume that $\beta_i$ is $G_i$-invariant. 

We can write $${\omega_i}_{|X_{ij}} = \phi_{ij}\, {\omega_j}_{|X_{ij}}$$ on 
$X_{ij}$ where $\phi_{ij}$ is a nowhere vanishing holomorphic function since the holomorphic 
$1$-forms ${\omega_i}_{|X_{ij}}$ and ${\omega_j}_{|X_{ij}}$ both define the foliation $p_{ij,i}^{-1}p_i^{-1}\sL=p_{ij,j}^{-1}p_j^{-1}\sL$ on $X_{ij}$ and have isolated zeros. 
The holomorphic functions $(\phi_{ij}\circ q_{ij})^m$ on $\wh{X}_{ij}$ then give a cocycle with respect to the open covering
$(g\cdot \wh{X}_i)_{g \in G, i \in I}$ that corresponds to the isomorphism class of $\wh{\sN}_\sL^{\,\,\otimes m}$ as a complex analytic line bundle since it does over the open set $X_\textup{reg}$. Notice that the smooth form $\frac{d\phi_{ij}}{\phi_{ij}}+{\beta_j}_{|X_{ij}}-{\beta_i}_{|X_{ij}}$ vanishes identically if $i=j$.
An easy computation now shows that 
$$\Big(\frac{d\phi_{ij}}{\phi_{ij}}+{\beta_j}_{|X_{ij}}-{\beta_i}_{|X_{ij}}\Big)\wedge {\omega_i}_{|X_{ij}}=0$$
if $i\neq j$ using the fact that $W_i \cap V_j =\emptyset$. 
Therefore, the cocycle of smooth $(1,0)$-forms $\Big(\frac{d\phi_{ij}}{\phi_{ij}}+{\beta_j}_{|X_{ij}}-{\beta_i}_{|X_{ij}}\Big)_{ij}$ can be viewed as a cocycle of smooth sections of the $\mathbb{Q}$-line bundle $\sM_\mathbb{Q}$ induced by $\sN_{\sL}^*$ on $X_\mathbb{Q}$. 
On the other hand, the smooth $(1,0)$-form $m\Big(\frac{d\phi_{ij}}{\phi_{ij}}+{\beta_j}_{|X_{ij}}-{\beta_i}_{|X_{ij}}\Big)$ is the pull-back of a smooth $(1,0)$-form on $V_i \cap V_j$ for $i\neq j$ and vanishes identically if $i=j$. Using a partition of unity subordinate to the open cover 
$(V_i)_{i \in I}$
given by \cite[Proposition 1.2]{chiang}, we see that we may assume that there exist smooth $(1,0)$-forms $\gamma_i$ on $X_i$ such that 
$$\gamma_i \wedge \omega_i=0\textup{ on  }X_i,$$
\noindent and 
$$\frac{d\phi_{ij}}{\phi_{ij}}=\beta_i-\beta_j + \gamma_i - \gamma_j\textup{ on  }X_{ij}.$$
Notice that we have $d\omega_i = (\beta_i+\gamma_i)\wedge \omega_i$ on $X_i \setminus p_i^{-1}(\wb{W}_i)$. The
$2$-form $\Omega$ defined on $X_i$ by $\Omega_{|X_i}:=\frac{1}{2i\pi}d(\beta_i+\gamma_i)$ is a well-defined closed $2$-form on $X_\mathbb{Q}$ whose pull-back to any resolution of $\wh X$ is smooth
and represents the first Chern class of the pull-back of $\wh{\sN}_\sL$ to this resolution. It follows from the projection formula that 
$$\deg p\cdot c_1(\sN_\sL)^2 = c_1(\wh{\sN}_\sL)^2 = \int_{\wh X} \wh{\Omega}\wedge \wh{\Omega},$$
where $\wh{\Omega}$ denotes the $2$-form induced by $\Omega$ on $\wh{X}$.
On the other hand, on $\wh{X}_i\setminus p^{-1}(\wb{W}_i)$, we have $\wh{\Omega}\wedge \wh{\Omega} = 0$ by construction. Stokes' Theorem then implies that  
$$c_1(\sN_\sL)^2 = \sum_i \frac{1}{(2i\pi)^2\deg q'_i}\int_{\partial p_i^{-1}(\wb{W}'_i)}(\beta_i+\gamma_i)\wedge d(\beta_i+\gamma_i),$$
where $W_i \Subset W'_i \Subset V_i$ is the image in $V_i$ of some small enough open ball in $X_i$. We finally obtain
$$c_1(\sN_\sL)^2 = \sum_x \textup{BB}^\mathbb{Q}(\sL,x),$$
completing the proof of the proposition.
\end{proof}

\subsection{Camacho-Sad formula} We finally observe that the Camacho-Sad formula also extends to surfaces with quotient singularities.

\begin{defn}Let $X$ be a normal quasi-projective algebraic surface with quotient singularities, and let $\sL \subset T_X$ be a foliation of rank one. Let 
$C \subset X$ be a complete curve. Suppose that $C$ is invariant under $\sL$. Given $x \in X$, there exist an open analytic neighbourhood $U$ of $x$, a (not necessarily connected) smooth analytic complex manifold $V$ and a finite Galois holomorphic map $p\colon  V\to U$ that is \'etale outside of the singular locus. Set
$$\textup{CS}^\mathbb{Q}(\sL,C,x):=\frac{1}{\deg p}\sum_{y \in p^{-1}(x)} \textup{CS}\big(p^{-1}\sL_{|U},p^{-1}(C\cap U),y\big)$$ 
where $\textup{CS}\big(p^{-1}\sL_{|U},p^{-1}(C\cap U),y\big)$ denotes the Camacho-Sad index (we refer to \cite[Section 3.2]{brunella} for this notion).
\end{defn}

\begin{rem}
One readily checks that $\textup{CS}^\mathbb{Q}(\sL,C,x)$ is independent of the local chart $p \colon V \to U$ at $x$.
\end{rem}

\begin{prop}\label{prop:CS_orbifold}
Let $X$ be a normal quasi-projective algebraic surface with quotient singularities, and let $\sL \subset T_X$ be a foliation of rank one. Let $C \subset X$ be a complete curve and suppose that $C$ is invariant under $\sL$.
Then $$C^2 =\sum_{x\in C} \textup{CS}^\mathbb{Q}(\sL,C,x).$$
\end{prop}

\begin{proof}
The proof is very similar to that of \cite[Theorem 3.2]{brunella} and so we again leave some easy details to the reader.
We maintain notation as in the proof of Proposition \ref{prop:BB_orbifold}. In particular, there is a well-defined closed $2$-form $\Omega$
on $X_\mathbb{Q}$ whose pull-back to any resolution of $\wh X$ is smooth
and represents the first Chern class of the pull-back of $\wh{\sN}_\sL$ to this resolution.

Shrinking the $V_i$, is necessary, we may assume that the ideal sheaf $p_i^{-1}\sI_C$ is generated by a holomorphic function $f_i$.  
Notice that the corresponding analytic $\mathbb{Q}$-sheaf $\sO_X(-C)_\mathbb{Q}$ is a $\mathbb{Q}$-line bundle. Let
$\wh{C}$ be the complete curve on $\wh{X}$ whose ideal sheaf is 
$\sO_{\wh X}\big(-\wh{C}\,\big):=p^{[*]}\sI_C$. 
Set $\alpha_i:=f_i^{-1}\omega_i$. 
The method used to construct the $2$-form $\Omega$ from the $\omega_i$ in the proof of Proposition \ref{prop:BB_orbifold} 
gives a closed $2$-form $A$   
on $X_\mathbb{Q}$ whose pull-back to any resolution of $\wh X$ is smooth
and represents the first Chern class of the pull-back of the line bundle $\wh{\sN}_\sL\otimes \sO_{\wh X}\big(-\wh{C}\,\big)$ to this resolution and such that 
$A_{|X_i}=\frac{1}{2i\pi}d \mu_i$, where $\mu_i$ is a smooth $(1,0)$-form such that
$d\alpha_i=\mu_i \wedge \alpha_i$ on $X_i \setminus p_i^{-1}(\wb{W}_i)$. 

Suppose that $p_i^{-1}\sL$ is singular, and let $x \in W_i$ be its singular point. There exist holomorphic functions $g_i$ and $h_i$ 
and a holomorphic $1$-form $\eta_i$ such that $g_i\omega_i = h_idf_i + f_i \eta_i$ in some analytic open neighbourhood of $x$. Then
$$\textup{CS}\big(p_i^{-1}\sL,p_i^{-1}(C),x\big)= - \frac{1}{2i\pi}\int_{\sigma_i} h_i^{-1}\eta_i,$$
where $\sigma_i \subset C$ is a union of small suitably oriented circles centered at $x$, one for each local branch of $C$ at $x$.
In particular, $\int_{\sigma_i} h_i^{-1}\eta_i$ depends only on $p_i^{-1}\sL$, $p_i^{-1}(C)$ and $x$. A straightforward computation then shows that
$d\alpha_i=(h_i^{-1}\eta_i+\beta_i+\gamma_i)\wedge \alpha_i$. On the other hand, we may obviously assume that $f_i$ and $h_i$ are relatively prime. Then, we must have ${\mu_i}_{|\sigma_i}={h_i^{-1}\eta_i}_{|\sigma_i}+{\beta_i}_{|\sigma_i}+{\gamma_i}_{|\sigma_i}$.

Let $\wh{A}$ be the $2$-form induced by $A$ on $\wh{X}$. Then
$$C^2=\frac{1}{\deg p}\wh{C}^2=\frac{1}{\deg p}\int_{\wh{C}}\wh{\Omega}-\frac{1}{\deg p}\int_{\wh{C}}\wh{A}.$$
On the other hand, by Stokes' theorem, we have
$$\frac{1}{\deg p}\int_{\wh{C}}\wh{\Omega}-\frac{1}{\deg p}\int_{\wh{C}}\wh{A}
=\sum_i\frac{1}{2i\pi\deg q'_i}\int_{\sigma_i}(\beta_i+\gamma_i-\mu_i)
=\sum_{x\in C} \textup{CS}^\mathbb{Q}(\sL,C,x).$$
This finishes the proof of the proposition.
\end{proof}

\section{Foliations on Mori fiber spaces}

In this section we provide another technical tools for the proof of the main results.

\begin{prop}\label{prop:transverse}
Let $X$ be a projective variety with klt singularities and let $\psi\colon X \to Y$ be a Mori fiber space with $\dim Y = \dim X -1\ge 1$. Let $\sG$ be a codimension one foliation on $X$ with $K_\sG\equiv 0$.
Suppose that there exist a closed subset $Z \subset X$ of codimension at least $3$
and an analytic quasi-\'etale $\mathbb{Q}$-structure $\{p_i:X_i\to X\setminus Z\}_{i \in I}$ on $X\setminus Z$ such that 
$p_i^{-1}\sG_{|X\setminus Z}$ is defined by a closed $1$-form with zero set of codimension at least two or by the $1$-form $x_{2}dx_{1} + \lambda_i x_{1}dx_{2}$ where $(x_{1},\ldots,x_{n})$ are analytic coordinates on $X_i$ and $\lambda_i \in \mathbb{Q}_{>0}$. Then there exists an open subset $Y^\circ \subseteq Y_{\textup{reg}}$ with complement of codimension at least two and a finite Galois cover $g\colon T \to Y$ such that the following holds. Set $T^\circ:=g^{-1}(Y^\circ)$ and $X^\circ:=\psi^{-1}(Y^\circ)$. 
\begin{enumerate}
\item The variety $T$ has canonical singularities and $K_T \sim_\mathbb{Z}0$; $T^\circ$ is smooth.
\item The normalization $M^\circ$ of the fiber product $T^\circ\times_Y X$ is a $\mathbb{P}^1$-bundle over $T^\circ$ and the map $M^\circ \to X^\circ$ is a quasi-\'etale cover.
\item The pull-back of $\sG_{|X^\circ}$ on $M^\circ$ yields a flat Ehresmann connection on $M^\circ \to T^\circ$.
\end{enumerate}
\end{prop}

Before we give the proof of Proposition \ref{prop:transverse}, we need the following auxiliary result, which might be of independent interest.

\begin{lemma}\label{lemma:P_bundle}
Let $X$ be quasi-projective variety with quotient singularities and let $\psi\colon X \to Y$ be a dominant projective morphism with connected and reduced fibers onto a smooth quasi-projective variety $Y$. Suppose that $\psi$ is equidimensional with $\dim Y = \dim X -1\ge 1$. Suppose in addition that $-K_X$ is $\psi$-ample. Let $\sG$ be a codimension one foliation on $X$ with $K_\sG\equiv_\psi 0$.
Suppose furthermore that there exists an analytic quasi-\'etale $\mathbb{Q}$-structure $\{p_i:X_i\to X\}_{i \in I}$ on $X$ such that $p_i^{-1}\sG$ is defined by a closed $1$-form with zero set of codimension at least two or by the $1$-form $x_{2}dx_{1} + \lambda_i x_{1}dx_{2}$ where $(x_{1},\ldots,x_{n})$ are analytic coordinates on $X_i$ and $\lambda_i \in \mathbb{Q}_{>0}$. Then there exists an open subset $Y^\circ \subseteq Y$ 
with complement of codimension at least two such that $\psi^\circ:=\psi_{|X^\circ}$ is a $\mathbb{P}^1$-bundle where $X^\circ:=\psi^{-1}(Y^\circ)$.
Moreover, $\sG_{|X^\circ}$ yields a flat Ehresmann connection on $\psi^\circ$.
\end{lemma}

\begin{proof} Notice that $\sG$ is regular along the generic fiber of $\psi$. It follows that
general fibers of $\psi$ are not tangent to $\sG$ since $-K_X$ is $\psi$-ample and $K_\sG\equiv_\psi 0$ by assumption (see Lemma \ref{lemma:residue_orbifold}). We first show the following. 

\begin{claim}\label{claim:irreducibility}
There exists an open set $Y^\circ$ with complement of codimension at least two such that, for any $y \in Y^\circ$, $\psi^{-1}(y)$ is transverse to $\sG$ at some point on $\psi^{-1}(y)$.
\end{claim}

\begin{proof} We argue by contradiction and assume that there is an irreducible hypersurface $D$ in $Y$ such that any irreducible component of $\psi^{-1}(y)$ is 
tangent to $\sG$ for all $y \in D$.

\medskip

\noindent\textit{Step 1. Setup.} Let $B \subset Y$ be a one dimensional complete intersection of general members of a very ample linear system on $Y$ passing through a general point $y$ of $D$ and set $S =:\psi^{-1}(B) \subset X$. Notice that $S$ is klt. Let $\sL$ be the foliation 
of rank one on $S$ induced by $\sG$, and denote by $\pi \colon S \to B$ the restriction of $\psi$ to $S$.
By the adjunction formula, $-K_S={-K_X}_{|S}$ is $\pi$-ample. 
By general choice of $B$, we have $\det\sN_\sL \cong (\det\sN_\sG)_{|S}$ (see Lemma \ref{lemma:bertini}), and hence $K_\sL\equiv_\pi 0$. Set $C:=\psi^{-1}(y) \subset S$ and observe that $C^2=0$ by construction.

By general choice of $B$, we may also assume that the collection of charts $\{p_i\colon X_i\to X\}_{i \in I}$ induces an analytic quasi-\'etale $\mathbb{Q}$-structure  
$\{q_i:S_i\to S\}_{i \in I}$ on $S$ where $S_i:=p_i^{-1}(S)$ and $q_i:={p_i}_{|S_i}$, and that either $q_i^{-1}\sL$ is defined by the $1$-form
$d f_i$ where $f_i$ is a holomorphic function on $S_i$ such that $d f_i$ has isolated zeroes (see Lemma \ref{lemma:bertini}) or it is given by the local $1$-form $z_{2}dz_{1} + \lambda_i z_{1}dz_{2}$ where $(z_{1},z_{2})$ are analytic coordinates on $S_i$ and $\lambda_i \in \mathbb{Q}_{>0}$.

\medskip

\noindent\textit{Step 2.} By assumption, any irreducible component of $C$ is invariant under $\sL$. 
Let $x_i$ be a point on $S_i$ with $p_i(x_i)\in C$. 

Suppose that $q_i^{-1}\sL$ is defined by the $1$-form $d f_i$ where $f_i$ is a holomorphic function on $S_i$ such that $d f_i$ has isolated zeroes.
Suppose furthermore that $q_i^{-1}(C)$ is given at $x_i$ by equation $t_i=0$ and that $f_i(x_i)=0$. Then $f_i=t_i g_i$ for some local holomorphic function $g_i$ on $S_i$ at $x_i$ and $-\textup{CS}(q_i^{-1}\sL,q_i^{-1}(C),x_i)$ is equal to the vanishing order of ${g_i}_{|q_i^{-1}(C)}$ at $x_i$. 
In particular, we have $\textup{CS}(q_i^{-1}\sL,q_i^{-1}(C),x_i) \le 0$.

Suppose now that $q_i^{-1}\sL$ is given by the local $1$-form $z_{2}dz_{1} + \lambda_i z_{1}dz_{2}$ where $(z_{1},z_{2})$ are analytic coordinates on $S_i$ and $\lambda_i \in \mathbb{Q}_{>0}$. If $x_i$ is not a singular point of $q_i^{-1}\sL$ then $\textup{CS}(q_i^{-1}\sL,q_i^{-1}(C),x_i)=0$. Suppose otherwise.
Then 
$$ \textup{CS}(q_i^{-1}\sL,q_i^{-1}(C),x_i) = \left\{
\begin{array}{ll}
-\lambda_i^{-1}<0 & \text{if } q_i^{-1}(C)=\{z_1=0\},\\
-\lambda_i<0 & \text{if } q_i^{-1}(C)=\{z_2=0\},\\
2-\lambda_i-\lambda_i^{-1}=-\lambda_i(1-\lambda_i^{-1})^2 \leqslant 0 & \text{if } q_i^{-1}(C)=\{z_1z_2=0\}.
\end{array}
\right. $$

Together with the Camacho-Sad formula (see Proposition \ref{prop:CS_orbifold}) and using $C^2=0$, this shows that $\lambda_i=1$ for all $i\in I$ as above. 

\medskip

\noindent\textit{Step 3.} By Step 2, for any $i \in I$ such that $p_i(X_i) \cap C \neq \emptyset$, $p_i^{-1}\sG$ is defined by a closed $1$-forms with zero set of codimension at least two. But then Lemma \ref{lemma:residue_orbifold} yields a contradiction since $c_1(\sN_\sL)\cdot C = 2$.
This finishes the proof of the claim.
\end{proof}

Next, we show that, shrinking $Y^\circ$, if necessary, $\psi^{-1}(y)$ is irreducible for every $y \in Y^\circ$.
We argue by contradiction again and assume that there is an irreducible hypersurface $D$ in $Y$ such that $\psi^{-1}(y)$ is reducible for a general point $y \in D$.
We maintain notation of Step 1 of the proof of Claim \ref{claim:irreducibility}. By Claim \ref{claim:irreducibility}, there is
an irreducible component $C_1$ of $C$ which is not invariant under $\sL$. Since $C$ is reducible by assumption, we must have $C_1^2<0$. On the other hand, by Lemma \ref{lemma:log_canonical_invariant_curve}, we have $C_1^2 = K_\sL \cdot C_1 + C_1^2\ge 0$, yielding a contradiction. This shows that $\psi$ has irreducible fibers at codimension one points in $Y$.

Applying \cite[Theorem II.2.8]{kollar96}, we conlcude that $\psi^\circ:=\psi_{|X^\circ}$ is a $\mathbb{P}^1$-bundle. By choice of $Y^\circ$, $F:=\psi^{-1}(y)\cong\mathbb{P}^1$ is not tangent to $\sG$ for every $y \in Y^\circ$. Since $c_1(\sN_\sL)\cdot F =2$, we see that $\sG$ is transverse to $\psi$ along $F$, completing the proof of the lemma.
\end{proof}

\begin{proof}[Proof of Proposition \ref{prop:transverse}] We maintain notation and assumptions of Proposition \ref{prop:transverse}.
Let $(D_j)_{j\in J}$ be the possibly empty set of hypersurfaces $D$ in $Y$ such that $\psi^*D$ is not integral. Since $\psi$ is a Mori fiber space, we must have 
$\psi^*D_j=m_jG_j$ for some integer $m_j\ge 2$ and some prime divisor $G_j$. Let $U_j \subseteq Y$ be a Zariski open neighbourhood of the generic point of $D_j$ and let $g_j\colon V_j \to U_j$ be a cyclic cover that branches along $D_j\cap U_j$ with ramification index $m_j$.
The normalization $M_j$ of the fiber product $V_j\times_Y X$ has geometrically reduced fibers over general points of 
$g_j^{-1}(D_j\cap U_j)$ by \cite[Th\'eor\`eme 12.2.4]{ega28}. Applying Lemma \ref{lemma:P_bundle} to $\psi_j\colon M_j \to V_j$, we see that $\psi_j$ is a $\mathbb{P}^1$-bundle over a Zariski open set in $V_j$ whose complement has codimension at least two. Moreover, 
$\sG_{|M_j}$ induces a flat Ehresmann connection on this $\mathbb{P}^1$-bundle.
In particular, there exists an open subset $Y^\circ \subseteq Y_{\textup{reg}}$ with complement of codimension at least two such that 
the restriction $\psi^\circ$ of $\psi$ to $X^\circ:=\psi^{-1}(Y^\circ)$ has irreducible fibers. Using \cite[Th\'eor\`eme 12.2.4]{ega28} together with Lemma \ref{lemma:P_bundle} again,
we see that we may also assume without loss of generality that $\psi$ is a $\mathbb{P}^1$-bundle over 
$Y^\circ\setminus \cup_{j\in J}D_j$ and that $\sG$ is everywhere transverse to $\psi$ over $Y^\circ\setminus \cup_{j\in J}D_j$.
Recall that $Y$ is $\mathbb{Q}$-factorial (see \cite[Lemma 5.1.5]{kmm}). Since $\textup{codim}\, Y \setminus Y^\circ \ge 2$ and $K_\sG\equiv 0$,
we must have 
$$K_Y+\sum_{i\in I}\frac{m_i-1}{m_i}D_i\equiv 0.$$
On the other hand, the proof of \cite[Corollary 4.5]{fujino99} shows that the pair 
$\big(Y,\sum_{i\in I}\frac{m_i-1}{m_i}D_i\big)$ is klt. Applying \cite[Corollary V.4.9]{nakayama04}, we conclude that $K_Y+\sum_{i\in I}\frac{m_i-1}{m_i}D_i$ is torsion. Let $g \colon T \to Y$ be the index one cover of the pair $\big(Y,\sum_{i\in I}\frac{m_i-1}{m_i}D_i\big)$ (see \cite[Section 2.4]{shokurov_log_flips}). By construction, we have 
$$K_T \sim_\mathbb{Q}g^*\Big(K_Y+\sum_{i\in I}\frac{m_i-1}{m_i}D_i\Big) \sim_\mathbb{Q} 0.$$
Replacing $T$ 
by a further quasi-\'etale cover, we may therefore assume that $K_T \sim_\mathbb{Z}0$. Then $T$ 
has canonical singularities. Shrinking $Y^\circ$, if necessary, we may assume that $T^\circ$ is smooth, that the
normalization $M^\circ$ of the fiber product $T^\circ\times_Y X$ has reduced and irreducible fibers over $T^\circ$ and that the map $M^\circ \to X^\circ$ is a quasi-\'etale cover. We may finally assume that
$M^\circ$ is a $\mathbb{P}^1$-bundle over $T^\circ$ and that
$\sG_{|M^\circ}$ yields a flat Ehresmann connection on $M^\circ \to T^\circ$ (see Lemme \ref{lemma:P_bundle}).
\end{proof}

\section{Algebraic integrability I}\label{section:algebraic_integrability_1}
 
The following is the main result of this section. We confirm the Ekedahl-Shepherd-Barron-Taylor conjecture (see \cite{esbt}) for mildly singular codimension one foliations with trivial canonical class on projective varieties with klt singularities and $\nu(X)=-\infty$.

\begin{thm}\label{thm:grothendieck_katz}
Let $X$ be a normal complex projective variety with klt singularities, and let $\sG$ be a codimension one foliation on 
$X$. Suppose that $\sG$ is canonical, and that it is closed under $p$-th powers for almost all primes $p$.
Suppose furthermore that $K_X$ is not pseudo-effective, and that $K_\sG\equiv 0$. 
Then $\sG$ is algebraically integrable. 
\end{thm}

\begin{proof} For the reader's convenience, the proof is subdivided into a number of steps. 

\medskip

\noindent\textit{Step 1.} Arguing as in Steps $1-2$ of the proof of \cite[Theorem 9.4]{cd1fzerocan}, we see that we may assume without loss of generality that $X$ is $\mathbb{Q}$-factorial and that there exists a Mori fiber space $\psi \colon X \to Y$.

\medskip

\noindent\textit{Step 2.} We first show that $\dim X - \dim Y = 1$. We argue by contradiction and assume
that $\dim X - \dim Y \ge 2$. Let $F$ be a general fiber of $\psi$. Note that $F$ has klt singularities, and that $K_F \sim_\mathbb{Z} {K_X}_{|F}$ by the adjunction formula. Moreover, $F$ is a Fano variety by assumption. 
Let $\sH$ be the foliation on $F$ induced by $\sG$. Since $c_1(\sN_\sG)\equiv -K_X$ is relatively ample, we see that $\sH$ has codimension one by Lemma \ref{lemma:residue_orbifold}.
By \cite[Proposition 3.6]{cd1fzerocan}, we have $K_\sH\sim_\mathbb{Z} {K_\sG}_{|F} - B$ for some effective Weil divisor $B$ on $F$. Suppose that $B\neq 0$.
Applying \cite[Theorem 4.7]{campana_paun19} to the pull-back of $\sH$ on a resolution of $F$, we see that $\sH$ is uniruled. This implies that $\sG$ is uniruled as well since $F$ is general. But this contradicts \cite[Proposition 4.22]{cd1fzerocan}, and shows that $B=0$. By \cite[Proposition 4.22]{cd1fzerocan} applied to $\sH$, we see that $\sH$ is canonical. Finally, one readily checks that $\sH$ is closed under $p$-th powers for almost all primes $p$.

Let $S \subseteq F$
be a two dimensional complete intersection of general elements of a very ample linear system $|H|$ on $F$. 
Notice that $S$ has klt singularities and hence quotient singularities.
Let $\sL$ be the foliation of rank one on $S$ induced by $\sH$.
By \cite[Proposition 3.6]{cd1fzerocan}, we have $\det\sN_\sL \cong (\det\sN_\sH)_{|S}$. It follows that 
\begin{equation}\label{BB_formula0}
c_1(\sN_\sL)^2 = K_F ^2 \cdot H^{\dim F -2} > 0.
\end{equation}

\bigskip

Recall from \cite[Proposition 9.3]{greb_kebekus_kovacs_peternell10} that there is an open set $F^\circ \subseteq F$ with quotient singularities whose complement in $F$ has codimension at least three. Let $\{p_\alpha\colon F_\alpha\to F^\circ\}_{\alpha \in A}$ be a quasi-\'etale $\mathbb{Q}$-structure on $F^\circ$ (see Fact \ref{fact:quotient_singularities_Q_strucutures}). Notice that $p_\alpha^{-1}\sH$ is obviously closed under $p$-th powers for almost all primes $p$. By \cite[Corollary 7.8]{lpt}, we see that we may assume that $p_\alpha^{-1}{\sH}_{|F^\circ}$ is defined at singular points (locally for the analytic topology) by the $1$-form $x_{2}dx_{1} + \lambda x_{1}dx_{2}$ where $(x_{1},\ldots,x_{n})$ are local analytic coordinates on $F_\alpha$ and $\lambda\in \mathbb{Q}_{>0}$.
By general choice of $S$, we may assume that $S \subset F^\circ$ and that the collection of charts $\{p_\alpha\colon F_\alpha\to F^\circ\}_{\alpha \in A}$ induces a quasi-\'etale $\mathbb{Q}$-structure $\{{p_\alpha}_{|S_\alpha}\colon S_\alpha \to S\}_{\alpha \in A}$ on $S$, where $S_\alpha:=p_\alpha^{-1}(S)$. We may assume in addition that
$({p_\alpha}_{|S_\alpha})^{-1}\sL$ is defined at a singular point $x$ (locally for the analytic topology) 
by a closed holomorphic $1$-form with isolated zeroes or
by the  $1$-form $x_{2}dx_{1} + \lambda x_{1}dx_{2}$ where $(x_{1},x_{2})$ are local analytic coordinates on $S_\alpha$ and $\lambda\in \mathbb{Q}_{>0}$. In the former case, we have $\textup{BB}\big(({p_\alpha}_{|S_\alpha})^{-1}\sL,x\big) = 0$, and in the latter case, we have 
$\textup{BB}\big(({p_\alpha}_{|S_\alpha})^{-1}\sL,x\big) = -\lambda(1-\lambda^{-1})^2$. 
Applying Proposition \ref{prop:BB_orbifold}, we obtain 
$$c_1(\sN_\sL)^2 =\sum_x \textup{BB}^\mathbb{Q}(\sL,x) \le 0.$$
But this contradicts inequality \eqref{BB_formula0} above, and shows that $\dim X - \dim Y = 1$.

\medskip

By \cite[Proposition 9.3]{greb_kebekus_kovacs_peternell10}, there exists a closed subset $Z \subset X$ of codimension at least $3$ such that $X \setminus Z$ has quotient singularities.
Let $\{p_\beta\colon X_\beta\to X\setminus Z\}_{\beta \in B}$ be a quasi-\'etale $\mathbb{Q}$-structure on $X\setminus Z$ (see Fact \ref{fact:quotient_singularities_Q_strucutures}). Notice that $p_\beta^{-1}\sG$ is obviously closed 
under $p$-th powers for almost all primes $p$. By \cite[Corollary 7.8]{lpt}, we see that there exists a closed subset $Z_\beta \subseteq X_\beta$ of codimension at least $3$ in $X_\beta$ such that ${p_\beta^{-1}\sG}_{|X_\beta\setminus Z_\beta}$ is defined at singular points (locally for the analytic topology) by the 
$1$-form $x_{2}dx_{1} + \lambda x_{1}dx_{2}$ where $(x_{1},\ldots,x_{n})$ are local analytic coordinates on $X_\beta$ and $\lambda\in \mathbb{Q}_{>0}$. Therefore, Proposition \ref{prop:transverse} applies. There exists an open subset $Y^\circ \subseteq Y_{\textup{reg}}$ with complement of codimension at least two and a finite cover $g\colon T \to Y$ such that the following holds. Set $T^\circ:=g^{-1}(Y^\circ)$ and $X^\circ:=\psi^{-1}(Y^\circ)$. 
\begin{enumerate}
\item The variety $T$ has canonical singularities and $K_T \sim_\mathbb{Z}0$; $T^\circ$ is smooth.
\item The normalization $M^\circ$ of the fiber product $T^\circ\times_Y X$ is a $\mathbb{P}^1$-bundle over $T^\circ$ and the map $M^\circ \to X^\circ$ is a quasi-\'etale cover.
\item The pull-back of $\sG_{|X^\circ}$ on $M^\circ$ yields a flat Ehresmann connection on $M^\circ \to T^\circ$.
\end{enumerate}

\noindent\textit{Step 3.} By \cite[Corollary 3.6]{gkp_bo_bo} applied to $T$, we see that there exists an abelian variety $A$ as well as
a projective variety $Z$ with $K_{Z}\sim_\mathbb{Z}0$ and augmented irregularity $\wt q(Z) = 0$ (we refer to \cite[Definition 3.1]{gkp_bo_bo} for this notion), and a quasi-\'etale cover $f\colon  A \times Z\to T$. 

Recall that $f$ branches only on the singular set of $T$, so that $f^{-1}(T^\circ)$ is smooth. On the other hand, 
since $f^{-1}(T^\circ)$ has complement of codimension at least two in $A\times Z_\textup{reg}$, we have 
$\pi_1\big(A\times Z_\textup{reg}\big) \cong \pi_1\big(f^{-1}(T^\circ)\big)$.
Now, consider the representation
$$\rho\colon \pi_1\big(A\times Z_\textup{reg}\big) \cong \pi_1\big(f^{-1}(T^\circ)\big) \to 
\pi_1\big(T^\circ\big) \to \textup{PGL}(2,\mathbb{C})$$
induced by $\sG_{|M^\circ}$. By \cite[Theorem I]{GGK}, the induced representation 
$$ \pi_1\big(Z_\textup{reg}\big) \to \pi_1\big(A\big)\times\pi_1\big(Z_\textup{reg}\big)\cong
\pi_1\big(A\times Z_\textup{reg}\big) \to \textup{PGL}(2,\mathbb{C})$$
has finite image.
Thus, replacing $Z$ by a quasi-\'etale cover, if necessary, we may assume without loss of generality that 
$\rho$ factors through the projection $\pi_1\big(A\times Z_\textup{reg}\big) \to \pi_1(A)$. 
Let $P$ be the corresponding $\mathbb{P}^1$-bundle over $A$. The natural projection $P \to A$ comes with a flat connection
$\sG_P \subset T_P$. By the GAGA theorem, $P$ is a projective variety. By assumption, its pull-back to 
$A\times Z_\textup{reg}$ agrees with $f^{-1}(T^\circ)\times_{T^\circ} M^\circ$ over $f^{-1}(T^\circ)$. Moreover, the pull-backs 
on $A\times Z_\textup{reg}$ of the foliations $\sG$ and $\sG_P$ agree as well, wherever this makes sense. In particular, $\sG$ is algebraically integrable if and only if so is $\sG_P$.
Now, one readily checks that $\sG_P$ is closed under $p$-th powers for almost all primes $p$.
Theorem \ref{thm:grothendieck_katz} then follows from \cite[Proposition 9.3]{cd1fzerocan}.
\end{proof}

\section{Algebraic integrability II}\label{section:algebraic_integrability_2}

In this section, we provide an algebraicity criterion for leaves of mildly singular codimension one algebraic foliations with numerically trivial canonical class on klt spaces $X$ with $\nu(X)=1$ (see Theorem \ref{thm:algebraic_integrability_nu_un}). We confirm the Ekedahl-Shepherd-Barron-Taylor conjecture in this special case.

\medskip

We will need the following auxiliary result, which might be of independent interest.

\begin{prop}\label{prop:general_type}
Let $X$ be a projective variety with klt singularities and let $\beta\colon Z \to X$ be a resolution of singularities. 
Let also $\phi \colon Z \to \mathfrak{H}:=\mathbb{D}^N/\Gamma$ be a generically finite morphism to a quotient of the polydisc $\mathbb{D}^N$ with $N \ge 2$ by 
an arithmetic irreducible lattice in $\Gamma \subset \textup{PSL}(2,\mathbb{R})^N$. Let $\sH$ be a codimension one foliation on $\mathfrak{H}$ induced by one of the tautological foliations on $\mathbb{D}^N$, and denote by $\sG$ the induced foliation on $X$. 
Suppose that $\phi(Z)$ is not tangent to $\sH$. Suppose furthermore that there exists an open set $X^\circ \subseteq X$ with complement of codimension at least $3$
and an analytic quasi-\'etale $\mathbb{Q}$-structure $\{p_i:X_i\to X^\circ\}_{i \in I}$ on $X^\circ$ such that 
either $p_i^{-1}\sG$ is regular or it is given by the local $1$-form $d(x_{1}x_{2})$ where $(x_{1},\ldots,x_{n})$ are analytic coordinates on $X_i$. Then $X$ is of general type.
\end{prop}

\begin{proof} For the reader's convenience, the proof is subdivided into a number of steps. By a result of Selberg, there exists a torsion-free subgroup 
$\Gamma_1$ of $\Gamma$ of finite index. Set $\mathfrak{H}_1:=\mathbb{D}^N/\Gamma_1$, and denote by $\pi \colon \mathfrak{H}_1 \to \mathfrak{H}$ the natural finite morphism. Recall that $\mathfrak{H}$ has isolated quotient singularities. It follows that $\pi$ is a quasi-\'etale cover since $N \ge 2$. One readily checks that there is only one separatrix for $\sH$ at any (singular) point.

\medskip

\noindent\textit{Step 1.} 
We first show the following.

\begin{claim}\label{claim:regular}
The rational map $\phi\circ\beta^{-1}\colon X \map \mathfrak{H}$ is a well-defined morphism over
$X^\circ$.
\end{claim}

\begin{proof}

Let $Z_1$ be the normalization of $Z\times_{\mathfrak{H}} \mathfrak{H}_1$, and let $X_1$ be the normalization of $X$ in the function field of $Z_1$.

\begin{center}
\begin{tikzcd}[row sep=large, column sep=large]
X_1\ar[d, "{f_1}"'] && Z_1 \ar[ll, "{\beta_1,\textup{ birational}}"'] \ar[d, "{g_1}"]\ar[rrr, "{\phi_1,\textup{ generically finite}}"] & & & \mathfrak{H}_1\ar[d, "{\pi,\textup{ quasi-\'etale}}"]\\
X && Z \ar[ll, "{\beta,\textup{ birational}}"]\ar[rrr, "{\phi,\textup{ generically finite}}"']&&& \mathfrak{H}.
\end{tikzcd}
\end{center}
Let $(D_j)_{j\in J}$ be the set of codimension one irreducible components of the branched locus of $f_1$. Observe that $\phi$ maps
$\beta_*^{-1}(D_j)$ to a (singular) point. It follows that $D_j$ is invariant under $\sG$ since there is only one separatrix for $\sH$ at any point and 
$\phi(Z)$ is not tangent to $\sH$ by assumption.
Let $\wh{X}_1 \to X_1$ be a finite cover such that the induced cover $p \colon  \wh{X}_1\to X$ is Galois. We may assume without loss of generality that 
$p$ is quasi-\'etale away from the branch locus of $f_1$. Therefore, there exist positive integers $(m_j)_{j\in J}$ 
such that 
$$K_{\wh{X}_1}= p^*\Big(K_X+\sum_{j\in J}\frac{m_j-1}{m_j}D_j\Big).$$ 
By Lemma \ref{lem:klt-pair} below, the pair $\big(X^\circ, \sum_{j\in J}\frac{m_j-1}{m_j}{D_j}_{|X^\circ}\big)$ is klt. 
It follows that $\wh{X}_1^\circ:=p^{-1}(X_1^\circ)$ has klt singularities as well.

Suppose that the rational map $\phi\circ\beta^{-1}\colon X \map \mathfrak{H}$ is not a well-defined morphism on $X^\circ$. By the rigidity lemma, there exist $x\in X^\circ$ such that $\dim \phi\big(\beta^{-1}(x)\big) \ge 1$.
Let $\wh{Z}_1$ be the normalization of $\wh{X}_1 \times_{X_1} Z_1$, and denote by 
$\wh{\beta}_1\colon \wh{Z}_1 \to \wh{X}_1$ and $\wh{\phi}_1\colon \wh{Z}_1 \to \mathfrak{H}_1$ the natural morphisms. Then there is a point
$x_1 \in \wh{X}_1^\circ:=p^{-1}(X_1^\circ)$ with $p(x_1)=x$ such that $\dim \wh{\phi}_1\big(\wh{\beta}_1^{-1}(x)\big) \ge 1$. 
On the other hand, $\wh{\beta}_1^{-1}(x)$ is rationally chain connected since 
$\wh{X}_1^\circ$ has klt singularities (see \cite[Corollary 1.6]{hacon_mckernan}). This yields a contradiction since 
$\mathfrak{H}_1$ is obviously hyperbolic. This finishes the proof of the claim.
\end{proof}

\noindent\textit{Step 2.} Let $F \subseteq Z$ be a prime divisor which is not $\beta$-exceptional and set $G=\beta(F)$.
We show that $\dim \phi(F) \ge 1$. We argue by contradiction and assume that $\dim \phi(F) =0$. Let $S \subseteq X$ be a two dimensional complete intersection of general elements of a very ample linear system on $X$. We may assume without loss of generality that $S$ is contained in $X^\circ$ and that it has klt singularities. Set $C:=S\cap G$. By Step 1, the rational map $\phi\circ\beta^{-1}$ is a well-defined morphism in a neighbourhood of $S$. But then $C^2<0$ since $C$ is contracted by the generically finite morphism
${\phi\circ\beta^{-1}}_{|S}\colon S\to \mathfrak{H}$. On the other hand, arguing as in Step 1, we see that $G$ must be invariant under $\sG$. Applying 
Lemma \ref{lemma:residue_orbifold}, we see that $C^2=C \cdot G =0$, yielding a contradiction.

\medskip

\noindent\textit{Step 3.} We use the notation of Step 1. Recall that $\pi\colon \mathfrak{H}_1  \to \mathfrak{H}$ is \'etale away from finitely many points.  
By Step 2, the natural map $Z\times_\mathfrak{H} \mathfrak{H_1} \to Z$ is a quasi-\'etale cover away from the exceptional locus of $\beta$. 
This immediately implies that $f_1$ is a quasi-\'etale cover. Let $Z_2\to Z_1$ be a resolution of singularities. 
We obtain a diagram as follows:

\begin{center}
\begin{tikzcd}[row sep=large, column sep=large]
X_1\ar[d, "{f_1, \textup{ quasi-\'etale}}"'] && Z_2 \ar[ll, "{\beta_2,\textup{ birational}}"'] \ar[d, "{g_2}"]\ar[rrr, "{\phi_2,\textup{ generically finite}}"] & & & \mathfrak{H}_1\ar[d, "{\pi,\textup{ quasi-\'etale}}"]\\
X && Z \ar[ll, "{\beta,\textup{ birational}}"]\ar[rrr, "{\phi,\textup{ generically finite}}"']&&& \mathfrak{H}.
\end{tikzcd}
\end{center}
Let $\sE_k$ ($1 \le k \le N$) be the codimension one regular foliations on $\mathfrak{H}_1$ induced by the tautological foliations 
on $\mathbb{D}^N$ so that $\Omega^1_{\mathfrak{H}_1}\cong \bigoplus_{1\le k \le N} \sN^*_{\sE_i}$. 
Set $n:=\dim X$. We may assume without loss of generality that the natural map 
$\bigoplus_{1\le k \le n} \phi_2^*\sN^*_{\sE_i} \to \Omega_{Z_2}^1$ is generically injective.
Now observe that the line bundle $\sN_{\sE_i}^*$ is hermitian semipositive, so that 
$\phi_2^*\sN^*_{\sE_i}$ is nef. On the other hand,
we have $c_1(\phi_2^*\sN^*_{\sE_1})\cdot\cdots\cdot c_1(\phi_2^*\sN^*_{\sE_n})>0$.
This immediately implies that $\kappa(Z_2)=\nu(Z_2)=\dim Z_2$. It follows that $\kappa(X_1)=\nu(X_1)=\dim X_1$ since $\beta_2$ is a birational morphism. Applying \cite[Proposition 2.7]{nakayama04}, we see that 
$\nu(X)=\nu(X_1)=\dim X_1=\dim X$ since $K_{X_1}\sim_\mathbb{Q}f_1^* K_X$. This completes the proof of the proposition.
\end{proof}

\begin{lemma}\label{lem:klt-pair}
Let $X$ be a variety of dimension $n$ with quotient singularities and let $\sG$ be a codimension one foliation on $X$. Let $\{p_i:X_i\to X\}_{i \in I}$ be an analytic quasi-\'etale $\mathbb{Q}$-structure on $X$. Suppose that 
either $p_i^{-1}\sG$ is regular or it is given by the $1$-form 
$x_{2}dx_{1} + \lambda_i x_{1}dx_{2}$ where $(x_{1},\ldots,x_{n})$ are analytic coordinates on $X_i$ and $\lambda_i \in \mathbb{Q}_{>0}$.
Let $(D_j)_{j\in J}$ be pairwise distinct prime divisors on $X$. Suppose that $D_j$ is invariant under $\sG$ for every $j \in J$. Then the pair $\big(X,\sum_{j\in J} a_jD_j\big)$ has klt singularities for any real numbers $0\leq a_j<1$.
\end{lemma}

\begin{proof}It suffices to prove the statement locally on $X$ for the analytic topology. 
We may therefore assume without loss of generality that there exist a (connected) smooth analytic complex manifold $Y$
and a finite Galois holomorphic map $p\colon  Y\to X$, totally branched over the singular locus and \'etale outside of the singular set.
We may also assume that $p^{-1}\sG$ is given by the $1$-form $y_{2}dy_{1} + \lambda y_{1}dy_{2}$, where $\lambda \in \mathbb{Q}_{>0}$ and $(y_{1},\ldots,y_{n})$ are analytic coordinates on $Y$. This immediately implies that the divisor 
$\sum_{j\in J} C_i$ has normal crossing support, where $C_j:=p^{-1}(D_j)$. Therefore,
the pair $\big(Y,\sum_{j\in J} a_jC_j\big)$ has klt singularities for any real numbers $0\leq a_j<1$.
On the other hand, we have 
$$K_Y+\sum_{j\in J}a_jC_j = p^*\big(K_X+\sum_{j\in J}a_jD_j\big)$$ 
and thus $\big(X,\sum_{j\in J}a_jD_j\big)$ has klt singularities as well. This finishes the proof of the lemma.
\end{proof}

The following is the main result of this section. 

\begin{thm}\label{thm:algebraic_integrability_nu_un}
Let $X$ be a normal projective variety with klt  singularities, and let $\sG$ be a codimension one foliation on $X$. Suppose that $\sG$ is canonical with $K_\sG\equiv 0$ and that $\nu(X)=1$. Suppose in addition that $\sG$ is closed under $p$-th powers for almost all primes $p$. Then $\sG$ is algebraically integrable. 
\end{thm}

\begin{proof}For the reader's convenience, the proof is subdivided into a number of steps. 

\medskip

\noindent\textit{Step 1.} Applying \cite[Proposition 8.14]{cd1fzerocan} together with Lemma \ref{lemma:canonical_quasi_etale_cover}, we may assume without loss of generality that there is no positive-dimensional algebraic subvariety tangent to $\sG$ passing through a general point of $X$. To prove the statement, it then suffices to show that $\dim X=1$.

\medskip

\noindent\textit{Step 2.} Arguing as in Steps $1$ and $2$ of the proof of \cite[Theorem 10.4]{cd1fzerocan}, we see that we may also assume 
that $X$ is $\mathbb{Q}$-factorial and that $K_X$ is movable.

\medskip

\noindent\textit{Step 3.} Recall from \cite[Proposition 9.3]{greb_kebekus_kovacs_peternell10} that there is an open set $X^\circ \subseteq X$ with quotient singularities whose complement in $X$ has codimension at least three. Let $\{p_\alpha\colon X_\alpha\to X^\circ\}_{\alpha \in A}$ be a quasi-\'etale $\mathbb{Q}$-structure on $X^\circ$ (see Fact \ref{fact:quotient_singularities_Q_strucutures}). Notice that $p_\alpha^{-1}\sG$ is obviously closed under $p$-th powers for almost all primes $p$. By \cite[Corollary 7.8]{lpt}, we see that we may assume that $p_\alpha^{-1}{\sG}_{|X^\circ}$ is defined at singular points (locally for the analytic topology) by the 
$1$-form $x_{2}dx_{1} + \lambda x_{1}dx_{2}$ where $(x_{1},\ldots,x_{n})$ are local analytic coordinates on $X_\alpha$ and $\lambda\in \mathbb{Q}_{>0}$.

Let $S \subseteq X$ be a two dimensional complete intersection of general elements of a very ample linear system $|H|$ on $F$. 
We may assume that $S \subset X^\circ$ and that $S$ has klt singularities. Let also $\sL$ be the foliation of rank one on $S$ induced by $\sG$.
By general choice of $S$, we may also assume that the collection of charts $\{p_\alpha\colon X_\alpha\to X^\circ\}_{\alpha \in A}$ induces a quasi-\'etale $\mathbb{Q}$-structure $\{q_\alpha\colon S_\alpha \to S\}_{\alpha \in A}$ on $S$, where $S_\alpha:=p_\alpha^{-1}(S)$ and $q_\alpha:={p_\alpha}_{|S_\alpha}$, and that $q_\alpha^{-1}\sL$ is defined at a singular point $x$ (locally for the analytic topology) 
by a closed holomorphic $1$-form with isolated zeroes or
by the  $1$-form $x_{2}dx_{1} + \lambda x_{1}dx_{2}$ where $(x_{1},x_{2})$ are local analytic coordinates on $S_\alpha$ and $\lambda\in \mathbb{Q}_{>0}$. In the former case, we have $\textup{BB}(q_\alpha^{-1}\sL,x) = 0$, and in the latter case, we have 
$\textup{BB}(q_\alpha^{-1}\sL,x) = -\lambda(1-\lambda^{-1})^2$. On the other hand, by \cite[Proposition 3.6]{cd1fzerocan}, we have $\det\sN_\sL \cong (\det\sN_\sG)_{|S}$. It follows that 
\begin{equation*}
c_1(\sN_\sL)^2 = K_X ^2 \cdot H^{\dim F -2} \ge 0
\end{equation*}
since $K_X$ is movable by Step 2. This implies that $\lambda=1$ by Proposition \ref{prop:BB_orbifold}.

\medskip

\noindent\textit{Step 4.} Let $\beta \colon Z \to X$ be a resolution of singularities with exceptional set $E$, and suppose that $E$ is a divisor with simple normal crossings. Suppose in addition that the restriction of $\beta$ to $\beta^{-1}(X_\textup{reg})$ is an isomorphism. 
Let $E_1$ be the reduced divisor on $Z$ whose support is
the union of all irreducible components of $E$ that are invariant under $\beta^{-1}\sG$.
Note that $-c_1(\sN_\sG)\equiv K_X$ by assumption. By Proposition \cite[Proposition 4.9]{cd1fzerocan} and \cite[Remark 4.8]{cd1fzerocan}, there exists a rational number $0 \le \varepsilon <1$ such that 
$$\nu\big(-c_1(\sN_{\beta^{-1}\sG})+\varepsilon E_1\big)=\nu\big(-c_1(\sN_\sG)\big)=1.$$ 
By \cite[Theorem 6]{touzet_conpsef} applied to $\beta^{-1}\sG$, we may assume that there exists an arithmetic irreducible lattice $\Gamma$ of $\textup{PSL}(2,\mathbb{R})^N$ for some integer $N \ge 2$, as well as a morphism $\phi\colon Z \to \mathfrak{H}:=\mathbb{D}^N/\Gamma$ of quasi-projective varieties such that $\sG=\phi^{-1}\sH$, where $\sH$ is a weakly regular codimension one foliation on  
$\mathfrak{H}$ induced by one of the tautological foliations on the polydisc $\mathbb{D}^N$. Note that $\phi$ is generically finite as there is no  positive-dimensional algebraic subvariety tangent to $\sG$ passing through a general point of $X$. Moreover, $\phi(Z)$ is not tangent to $\sH$ since $\sG$ has codimension one.

Proposition \ref{prop:general_type} then says that $X$ must be of general type. Thus $\dim X = \nu(X)=1$, completing the proof of the theorem.
\end{proof}

The same argument used in the proof of Theorem \ref{thm:algebraic_integrability_nu_un} shows that the following holds using the fact that weekly regular foliations on smooth spaces are regular (we refer to \cite[Section 5]{cd1fzerocan} for this notion). 

\begin{thm}\label{thm:algebraic_integrability_nu_un_regular}
Let $X$ be a normal projective variety with klt singularities, and let $\sG$ be a weakly regular codimension one foliation on $X$. Suppose that $\sG$ is canonical with $K_\sG\equiv 0$. Suppose in addition that $\nu(X)=1$. Then $\sG$ is algebraically integrable. 
\end{thm}

\section{Foliations defined by closed rational $1$-forms}

Let $X$ be a normal projective variety, and let $\sG\subset T_X$ be a codimension one foliation. Suppose that $\sG$ is given by a closed rational $1$-form $\omega$ with values in a flat line bundle $\sL$. Then the twisted rational $1$-form $\omega$ is not uniquely determined by $\sG$ in general. The following result addresses this issue.

\begin{prop}\label{prop:flatness_torsion}
Let $X$ be a normal projective variety with klt singularities, and let $\sG\subset T_X$ be a codimension one foliation with canonical singularities. Suppose that $\sG$ is given by a closed rational $1$-form 
$\omega$ with values in a flat line bundle $\sL$ whose zero set has codimension at least two. 
Suppose furthermore that $K_X$ is not pseudo-effective, and that $K_\sG\equiv 0$. 
Then there exists a quasi-\'etale cover $f\colon X_1 \to X$ such that $f^{-1}\sG$ is given by a closed
rational $1$-form with zero set of codimension at least two.
\end{prop}

To prove Proposition \ref{prop:flatness_torsion}, we will need the following auxiliary result.

\begin{lemma}\label{lemma:first_integral}
Let $X$ be a smooth quasi-projective variety, and let $\sG\subset T_X$ be a codimension one foliation with canonical singularities. Suppose that $\sG$ is given by a closed rational $1$-form $\omega$ with values in a flat line bundle whose zero set has codimension at least two. Suppose in addition that the residue of $\omega$ at a general point of any irreducible component of its polar set is zero. Let $B$ be a codimension two irreducible component of the singular set of $\sG$ which is contained in the polar set of $\omega$ and let $(D_i)_{i\in I}$ be the set of irreducible components of the polar set of $\omega$ containing $B$. If $x \in B$ is a general point, then $\sG$ has a first integral at $x$ of the form $f:=\prod_{i\in I}f_i^{m_i}$ where $f_i$ is a local holomorphic equation of $D_i$ at $x$
and $m_i$ is a positive integer for every $i\in I$.
\end{lemma}

\begin{proof}
Let $S \subseteq X$ be a two dimensional complete intersection of general elements of some very ample linear system on $X$. Notice that $S$ is smooth and that the foliation $\sL$ of rank one on $S$ induced by $\sG$ has canonical singularities in a Zariski open neighbourhood of $B \cap S$ by Proposition \ref{prop:bertini}. 

Let now $x \in B\cap S$. Let also $f_i$ be a local holomorphic equation of $D_i$ at $x$. By \cite[Th\'eor\`eme III.2.1, Premi\`ere partie]{cerveau_mattei}, there exist positive integers $m_i$ and a local holomorphic function $g$ at $x$ such that $\sG$ is given in some analytic open neighbourhood of $x$ by the $1$-form
$$d\Big(\frac{g}{\prod_{i\in I}f_i^{m_i}}\Big).$$ 
Set $f:=\prod_{i\in I}f_i^{m_i}$. We may also assume that $f_i$ and $g$ are relatively prime at $x$ for any $i \in I$. By general choice of $S$, the restrictions of $g$ and $f$ to $S$ are relatively prime non-zero local holomorphic functions. In particular, $\sL$ is given by $d\big(\frac{g}{f}\big)$ on some analytic open neighbourhood of $x$ in $S$. Now, recall from \cite[Observation I.2.6]{mcquillan08}, that $\sL$ has only finitely many separatrices at $x$. This immediately implies that $g(x)\neq 0$. Let $i_0 \in I$. Then $ \big(f_{i_0}g^{-\frac{1}{m_{i_0}}}\big)^{m_{i_0}}\cdot \prod_{i\in I\setminus\{i_0\}}f_i^{m_i}$ is a holomorphic
first integral of $\sG$ at $x$, proving the lemma.
\end{proof}

The proof of Proposition \ref{prop:flatness_torsion} makes use of the following minor generalization of \cite[Proposition 3.11]{lpt}. 

\begin{lemma}\label{lemma:closed_form_lc}
Let $X$ be a normal variety with quotient singularities and let $\sG$ be a codimension one foliation on $X$ with canonical singularities. Suppose that $\sG$ is given by a closed rational $1$-form $\omega$ with zero set of codimension at least two. Let $D$ be the reduced divisor on $X$ whose support is the polar set of $\omega$. Then $(X,D)$ has log canonical singularities. 
\end{lemma}

\begin{proof}
Let $\{p_{\alpha}:X_{\alpha}\to X\}_{\alpha \in A}$ be a quasi-\'etale $\mathbb{Q}$-structure on $X$. 
The claim follows easily from \cite[Proposition 3.11]{lpt} applied to $p_\alpha^{-1}\sG$ on $X_\alpha$ for any $\alpha\in A$ together with
\cite[Proposition 3.16]{kollar97}. 
\end{proof}

We will also need the following generalization of \cite[Corollary 3.10]{lpt}.

\begin{lemma}\label{lemma:leaf_rc_quotient}
Let $X$ be a normal projective variety with klt singularities, and let $\sG\subset T_X$ be a codimension one foliation with canonical singularities. Suppose that $K_\sG\equiv 0$. Suppose furthermore that $X$ is uniruled and let $X \map R$ be the maximal rationally chain connected fibration. Then any algebraic leaf of $\sG$ dominates $R$. 
\end{lemma}

\begin{proof}
Consider a commutative diagram:
\begin{center}
\begin{tikzcd}[row sep=large]
X_1  \ar[d, "{f_1}"']\ar[rr, "{\beta,\textup{ birational}}"] & & X\ar[d, dashed] \\
R_1 \ar[rr, "{\textup{ birational}}"]&& R,
\end{tikzcd}
\end{center}
where $X_1$ and $R_1$ are smooth projective varieties. Applying \cite[Theorem 1.1]{ghs03}, we see that $R_1$ is not uniruled. This in turn implies that $K_{R_1}$ is pseudo-effective by \cite[Corollary 0.3]{bdpp}. Set $\sM:=f_1^*\sO_{R_1}(-K_{R_1})$ and $q:=\dim R_1$, and let 
$\omega \in H^0\big(X_1,\Omega^q_{X_1}\otimes \sM\big)$ be the twisted $q$-form induced by $df_1$.
Let $\sG_1$ be the pull-back of $\sG$ on $X_1$. Then $\omega$ yields a non-zero section $\sigma \in H^0\big(X_1,\wedge^q(\sG_1^*)\otimes\sM\big)$ by 
\cite[Proposition 4.22]{cd1fzerocan}. On the other hand, $\wedge^q(\sG_1^*)$ is semistable with respect to the pull-back on $X_1$ of any ample divisor on $X$ by \cite[Lemma 8.15]{cd1fzerocan} together with \cite[Theorem 5.1]{campana_peternell11}. This immediately implies that codimension one zeroes of $\sigma$ are 
$\beta$-exceptional. Now, let $F$ be the closure of an algebraic leaf of $\sG$ and denote by $F_1$ its proper transform on $X_1$. If $F_1$ does not dominate $R_1$, then $\sigma$ must vanish along $F_1$, yielding a contradiction.
\end{proof}

\begin{proof}[{Proof of Proposition \ref{prop:flatness_torsion}}] For the reader's convenience, the proof is subdivided into a number of steps. By assumption, there exist prime divisors $(D_i)_{1\le i \le r}$ on $X$ and positive integers $(m_i)_{1\le i \le r}$ such that $\sN_\sG \cong \sO_X\big(\sum_{1\le i \le r} m_i D_i\big)\otimes\sL$.

\medskip

\noindent\textit{Step 1.} Let $X \map R$ be the maximal rationally chain connected fibration. Recall that it is an almost proper map and that its general fibers are rationally chain connected. Consider a commutative diagram:
\begin{center}
\begin{tikzcd}[row sep=large]
X_1  \ar[d]\ar[rr, "{\beta,\textup{ birational}}"] & & X\ar[d, dashed] \\
R_1 \ar[rr, "{\textup{ birational}}"]&& R,
\end{tikzcd}
\end{center}
where $X_1$ and $R_1$ are smooth projective varieties. 
Notice that general fibers of $X_1 \to R_1$ are rationally chain connected by \cite[Theorem 1.2]{hacon_mckernan} and
recall from \cite[Proposition III. 1.1, Premi\`ere partie]{cerveau_mattei}, that the hypersurfaces $D_i$ are invariant under $\sG$.
By Lemma \ref{lemma:leaf_rc_quotient} above, we conclude that $D_i$ dominates $R$ for every $1 \le i \le r$. Now, $\omega$ induces 
a closed rational $1$-form $\omega_1$ on $X_1$ with values in the flat line bundle $\beta^*\sL$.

Arguing as in \cite[Section 8.2.1]{lpt}, one concludes that $\beta^*\sL$ is torsion if the residue of
$\omega_1$ at a general point of $D_i$ is non-zero for some $1 \le i\le r$.

\medskip

Suppose from now on that the residue at a general point of $D_i$ is zero for any $1 \le i\le r$.

\medskip

\noindent\textit{Step 2.} Arguing as in Steps $1-2$ of the proof of \cite[Proposition 11.6]{cd1fzerocan}, we may assume without loss of generality that the following holds.

\begin{enumerate}
\item There is no positive dimensional algebraic subvariety tangent to $\sG$ passing through a general point in $X$.
\item The variety $X$ is $\mathbb{Q}$-factorial.
\end{enumerate}

\medskip

\noindent\textit{Step 3.} Since $K_X$ is not pseudo-effective by assumption, we may run a minimal model program for $X$ and end with a Mori fiber space (see \cite[Corollary 1.3.3]{bchm}). Therefore, there exists a sequence of maps

\begin{center}
\begin{tikzcd}[row sep=large, column sep=large]
X:=X_0 \ar[r, "{\phi_0}", dashrightarrow] & X_1 \ar[r, "{\phi_1}", dashrightarrow] & \cdots \ar[r, "{\phi_{i-1}}", dashrightarrow] & X_i \ar[r, "{\phi_i}", dashrightarrow] & X_{i+1} \ar[r, "{\phi_{i+1}}", dashrightarrow] & \cdots \ar[r, "{\phi_{m-1}}", dashrightarrow] & X_m \ar[d, "{\psi_m}"]\\
&&&&&& Y
\end{tikzcd}
\end{center}

\noindent where the $\phi_i$ are either divisorial contractions or flips, and $\psi_m$ is a Mori fiber space. The spaces $X_i$ are normal, $\mathbb{Q}$-factorial, and $X_i$ has klt singularities for all $0\le i \le m$. 
Let $\sG_i$ be the foliation on $X_i$ induced by $\sG$. Arguing as in Step 2 of the proof of 
\cite[Theorem 9.4]{cd1fzerocan}, we see that $K_{\sG_i}\equiv 0$ and that $\sG_i$ has canonical singularities.
Moreover, $\omega$ induces a closed rational $1$-form $\omega_m$ on $X_m$ with values in a flat line bundle $\sL_m$, whose zero set has codimension at least two. By construction, $\sG_m$ is given by $\omega_m$, and if $\sL_m$ is a torsion flat line bundle then so is $\sL$.

We will show in Steps $4-6$ that either $\sL_m$ is torsion, or $X_m$ is smooth, the polar locus of $\omega_m$ is a smooth connected hypersurface, say $D_{1,m}:=(\phi_{m-1}\circ\cdots\circ\phi_0)_* D_1$, and $\sG_m$ can be defined by a nowhere vanishing (closed) logarithmic $1$-form with poles along $D_{1,m}\sqcup D'_{1,m}$ for some smooth (connected) hypersurface $D'_{1,m}$. Taking this for granted, one concludes that 
the conclusion of Proposition \ref{prop:flatness_torsion} holds for $X$ as in Step 3 of the proof of \cite[Proposition 11.6]{cd1fzerocan}.

\medskip

For simplicity of notation, we will assume in the following that $X=X_m$, writing $\psi:=\psi_m$.

\medskip

\noindent\textit{Step 4.} Arguing as in Step 5 of the proof of \cite[Proposition 11.6]{cd1fzerocan}, we see that the following additional properties hold.

\begin{enumerate}
\item We have $\dim Y = \dim X -1$. 
\item Moreover, $\psi(D_1)=Y$, $m_1=2$ and $D_1 \cdot F = 1$, where $F$ denotes a general fiber of $\psi$. Moreover, 
$\psi(D_i) \subsetneq Y$ for any $2\le i \le r$.
\end{enumerate}

\medskip

\noindent\textit{Step 5.} Recall from \cite[Proposition 9.3]{greb_kebekus_kovacs_peternell10} that klt spaces have quotient singularities in codimension $2$. 

By Lemma \ref{lemma:closed_form_lc}, $\big(X,\sum_{1 \le i\le r}D_i\big)$ has log canonical singularities in codimension $2$.
Let $2\le i \le r$.
Lemma \ref{lemma:P_bundle_2} below then says that there exists a Zariski open set $Y_i^\circ$ containing the generic point of $\psi(D_i)$
and a finite cover $g_i^\circ\colon T_i^\circ \to Y_i^\circ$ of smooth varieties such that the normalization $M_i^\circ$ of the fiber product $T_i^\circ\times_Y X$ is a $\mathbb{P}^1$-bundle over $T_i^\circ$ and such that the map $M_i^\circ \to X_i^\circ$ is a quasi-\'etale cover, where $X_i^\circ:=\psi^{-1}(Y_i^\circ)$.
By Lemma \ref{lemma:canonical_quasi_etale_cover}, the pull-back $\sG_{M_i^\circ}$ of $\sG_{|X_i^\circ}$ to $M_i^\circ$ has canonical singularities. Applying Lemma \ref{lemma:first_integral} and shrinking $Y_i^\circ$, if necessary, we see that  
$\sG_{M_i^\circ}$ is defined locally for the analytic topology
by a closed $1$-form with zero set of codimension at least two or by the $1$-form $x_{2}dx_{1} + \lambda x_{1}dx_{2}$ where $(x_{1},\ldots,x_{n})$ are analytic coordinates on $X_i$ and $\lambda \in \mathbb{Q}_{>0}$. On the other hand, the restriction of $\sG$ to $X \setminus \cup_{1\le i \le r} D_i$ is locally defined by closed $1$-forms with zero set of codimension at least two by assumption.
Therefore, there exist a closed subset $Z \subset X$ of codimension at least $3$ and an analytic quasi-\'etale $\mathbb{Q}$-structure $\{p_j\colon X_j\to X\setminus Z\}_{j \in J}$ on $X\setminus Z$ such that 
$p_j^{-1}\sG$ is defined by a closed $1$-form with zero set of codimension at least two or by the $1$-form $x_{2}dx_{1} + \lambda_j x_{1}dx_{2}$ where $(x_{1},\ldots,x_{n})$ are analytic coordinates on $X_i$ and $\lambda_j \in \mathbb{Q}_{>0}$ (see Fact \ref{fact:purity}). 
Proposition \ref{prop:transverse} then says that there exists an open subset $Y^\circ \subseteq Y_{\textup{reg}}$ with complement of codimension at least two and a finite cover $g\colon T \to Y$ such that the following holds. Set $T^\circ:=g^{-1}(Y^\circ)$ and $X^\circ:=\psi^{-1}(Y^\circ)$. 
\begin{enumerate}
\item The variety $T$ has canonical singularities and $K_T \sim_\mathbb{Z}0$; $T^\circ$ is smooth.
\item The normalization $M^\circ$ of the fiber product $T^\circ\times_Y X$ is a $\mathbb{P}^1$-bundle over $T^\circ$ and the map $M^\circ \to X^\circ$ is a quasi-\'etale cover.
\item The pull-back of $\sG_{|X^\circ}$ yields a flat Ehresmann connection on $M^\circ \to T^\circ$.
\end{enumerate}

Moreover, there exists a $\mathbb{Q}$-divisor $B$ on $Y$ such that $K_Y+B$ is torsion (see the proof of Proposition \ref{prop:transverse}).

\medskip

\noindent\textit{Step 6.} Suppose from now on that $\sL$ is not torsion, and let $a_Y\colon Y \to \textup{A}(Y)$
be the Albanese morphism. Arguing as in Step $7$ of the proof of \cite[Proposition 11.6]{cd1fzerocan}, we conclude that $a_Y$ is generically finite.

If $B \neq 0$, then $Y$ is uniruled. But this contradicts the fact that $a_Y$ is generically finite. Therefore, $B=0$ and $K_Y$ is torsion.
The argument used in Step $7$ of the proof of \cite[Proposition 11.6]{cd1fzerocan} then shows that $X$ and $D_1$ are smooth and that 
$\sG$ can be defined by a nowhere vanishing (closed) logarithmic $1$-form with poles along $D_1\sqcup D'_1$ for some smooth (connected) hypersurface $D'_1$. 
This finishes the proof of the proposition.
\end{proof}

\begin{lemma}\label{lemma:P_bundle_2}
Let $X$ be a quasi-projective variety with klt singularities and let $\psi\colon X \to Y$ be a Mori fiber space where $Y$ is a smooth quasi-projective variety. Suppose in addition that $\dim Y = \dim X -1$ and that there exists a section $D \subset X$ of $\psi$. Let $G\subset Y$ be an irreducible hypersurface and set 
$E:=\psi^{-1}(G)$. Suppose furthermore that $(X,D+E)$ has log canonical singularities over the generic point of $G$. 
Then there exist a Zariski open neighbourhood $Y^\circ$ of the generic point of $G$ in $Y$ and a finite Galois cover $g^\circ\colon T^\circ \to Y^\circ$ of smooth varieties such that the normalization $M^\circ$ of the fiber product $T^\circ\times_Y X$ is a $\mathbb{P}^1$-bundle over $T^\circ$ and such that the map $M^\circ \to X^\circ$ is a quasi-\'etale cover, where $X^\circ:=\psi^{-1}(Y^\circ)$.
\end{lemma}

\begin{proof}By \cite[Proposition 9.3]{greb_kebekus_kovacs_peternell10}, klt spaces have quotient singularities in codimension two. Shrinking $Y$, if necessary, we may therefore assume that $X$ has quotient singularities. In particular, $X$ is $\mathbb{Q}$-factorial. Since $\psi$ is a Mori fiber space with $\mathbb{Q}$-factorial singularities and $D$ is a section of $\psi$, any irreducible component of $\psi^{-1}(y)$ meets $D$ at the unique intersection point of $D$ and $\psi^{-1}(y)$.
Let $B \subset Y$ be a one dimensional complete intersection of general members of a very ample linear system passing through a general point $y$ of $G$ and set $S =:\psi^{-1}(B) \subset X$. Notice that $S$ is klt and that $(S,D\cap S+E\cap S)$ is log canonical. This easily implies that the support of $D\cap S+E\cap S$ has at most two analytic branches at any point. 
This in turn implies that $\psi^{-1}(y)$ is irreducible and proves that $\psi$ has irreducible fibers in a neighbourhood of the generic point of $G$.

Write $\psi^*G=mE$ for some positive integer. Let $Y^\circ$ be a Zariski open neighbourhood of the generic point of $G$ in $Y$ and let $g^\circ\colon T^\circ \to Y^\circ$ be a finite Galois cover of smooth varieties with ramification index $m$ at any point over the generic point of $G$. Let $M^\circ$ be 
the normalization  of the fiber product $T^\circ\times_Y X$. Shrinking $Y^\circ$, if necessary, we may assume without loss of generality that the map 
$\phi^\circ\colon M^\circ \to T^\circ$ has reduced fibers (see \cite[Th\'eor\`eme 12.2.4]{ega28}). Observe that the map $f^\circ \colon M^\circ \to X^\circ$ is a quasi-\'etale cover, where $X^\circ:=\psi^{-1}(Y^\circ)$. It follows that $\big(M^\circ,(f^\circ)^{-1}(D\cap X^\circ)+(f^\circ)^{-1}(E\cap X^\circ)\big)$ is log canonical.
Moreover, $(f^\circ)^{-1}(D\cap X^\circ)$ is a section of $\phi^\circ$ and any irreducible component of any fiber of $\phi^\circ$ meets this section by construction. Arguing as above, we conclude that $\phi^\circ$ has irreducible fibers over general points in $\phi^\circ\big((f^\circ)^{-1}(E\cap X^\circ)\big)$. The lemma then follows from \cite[Theorem II.2.8]{kollar96}.
\end{proof}

\section{Proof of Theorem \ref{thm_intro:main} and Corollary \ref{cor_intro:main}}

\begin{proof}[{Proof of Theorem \ref{thm_intro:main}}]
We maintain notation and assumptions of Theorem \ref{thm_intro:main}. By \cite[Lemma 12.5]{cd1fzerocan}, we have $\nu(X)\le 1$.
If $\nu(X)=0$, then Theorem \ref{thm_intro:main} follows from \cite[Lemma 10.1]{cd1fzerocan} and \cite[Proposition 10.2]{cd1fzerocan}.

\medskip

We may therefore assume from now on that $\nu(X)\in\{-\infty,1\}$.

\medskip

By \cite[Proposition 12.3]{cd1fzerocan} together with \cite[Lemma 8.15]{cd1fzerocan}, either $\sG$ is closed under $p$-th powers for almost all primes $p$, or it is given by a closed rational $1$-form with values in a flat line bundle whose
zero set has codimension at least two (see also \cite[Remark 12.4]{cd1fzerocan}). 

\medskip

If $\sG$ is closed under $p$-th powers for almost all primes $p$, then it is algebraically integrable by Theorems \ref{thm:grothendieck_katz} and \ref{thm:algebraic_integrability_nu_un}. The statement then follows from \cite[Theorem 1.5]{cd1fzerocan}.

\medskip

Suppose now that $\sG$ is given by a closed rational $1$-form $\omega$ with values in a flat line bundle $\sL$ whose
zero set has codimension at least two. By assumption, there exists an effective divisor $D$ on $X$ such that $\sN_\sG\cong\sO_X(D)\otimes \sL$. On the other hand, $c_1(\sN_\sG)\equiv -K_X$ since $K_\sG \equiv 0$. This immediately implies that $\nu(X)=-\infty$. 
By Proposition \ref{prop:flatness_torsion}, there exists a quasi-\'etale cover $f\colon X_1 \to X$ such that $f^{-1}\sG$ is given by a closed
rational $1$-form with zero set of codimension at least two. Note that $X_1$ has klt singularities. Moreover, $f^{-1}\sG$ is canonical 
with $K_{f^{-1}\sG}\equiv 0$ by Lemma \ref{lemma:canonical_quasi_etale_cover}. 
Theorem \ref{thm_intro:main} follows from \cite[Theorem 11.3]{cd1fzerocan} in this case. This finishes the proof of Thereom \ref{thm_intro:main}.
\end{proof}

\begin{proof}[{Proof of Corollary \ref{cor_intro:main}}]
We maintain notation and assumptions of Corollary \ref{cor_intro:main}. By \cite[Theorem 8.1]{cascini_spicer} (see also \cite[Theorem 1.7]{spicer_svaldi}), there exists a projective birational morphism $\beta\colon X_1 \to X$ such that 
\begin{enumerate}
\item $X_1$ is $\mathbb{Q}$-factorial with klt singularities,
\item we have $K_{\beta^{-1}\sG}=\beta^*K_\sG$ and
\item $\sG_1:=\beta^{-1}\sG$ is F-dlt we refer to (\cite[Paragraph 3.2]{cascini_spicer} for this notion).
\end{enumerate}
By \cite[Proposition V.2.7]{nakayama04} and property (2) above, we have $\nu(K_{\sG_1})=\nu(K_\sG)=0$. Moreover, $\sG_1$ has canonical singularities by Lemma \ref{lemma:singularities_birational_morphism}.
Thus, we may run a minimal model program for $\sG_1$ and end with a minimal model. There exist a projective threefold $X_2$ with $\mathbb{Q}$-factorial klt singularities and a $K_{\sG_1}$-negative birational map
$\phi\colon X_1 \map X_2$ such that $K_{\sG_2}$ is nef, where
$\sG_2$ denotes the foliation on $X_2$ induced by $\sG_1$. 
Since $\phi$ is $K_{\sG_1}$-negative and using \cite[Proposition V.2.7]{nakayama04} again, we see that $\nu(K_{\sG_2})=0$. 
Since $K_{\sG_2}$ is nef, we must have $K_{\sG_2} \equiv 0$. 
On the other hand, using the fact that $\phi$ is $K_{\sG_1}$-negative together with Lemma \ref{lemma:singularities_birational_morphism}, we conclude that
$\sG_2$ has canonical singularities as well.
The corollary now follows from Theorem \ref{thm_intro:main}.
\end{proof}


\providecommand{\bysame}{\leavevmode\hbox to3em{\hrulefill}\thinspace}
\providecommand{\MR}{\relax\ifhmode\unskip\space\fi MR }
\providecommand{\MRhref}[2]{%
  \href{http://www.ams.org/mathscinet-getitem?mr=#1}{#2}
}
\providecommand{\href}[2]{#2}

\end{document}